\documentclass[reqno]{amsart}
\usepackage{amssymb, amsthm, amsmath, xypic, enumitem}
\usepackage[pagewise]{lineno}

\theoremstyle{definition}
\newtheorem{theorem}{Theorem}[section]
\newtheorem{definition}[theorem]{Definition}
\newtheorem{proposition}[theorem]{Proposition}

\newtheorem{corollary}[theorem]{Corollary}
\newtheorem{example}[theorem]{Example}
\newtheorem{remark}[theorem]{Remark}
\newtheorem{notation}[theorem]{Notation}

\usepackage{ifthen, xcolor}
\newboolean{commentBoolVar}
\setboolean{commentBoolVar}{false} 
\newcommand{\colourcomment}[3]
{%
\ifthenelse{\boolean{commentBoolVar}}{{\color{#2}(#1: #3)}}{}%
}%

\renewcommand{\bar}{\overline}
\renewcommand{\hat}{\widehat}
\renewcommand{\Im}{\mathrm{Im}}
\newcommand{\boldvarphi}{\boldsymbol{\varphi}}

\DeclareMathOperator{\Gal}{\mbox{Gal}}
\DeclareMathOperator{\Perm}{\mbox{Perm}}
\DeclareMathOperator{\Hol}{\mbox{Hol}}
\DeclareMathOperator{\Aut}{\mbox{Aut}}

\DeclareMathOperator{\Stab}{\mbox{Stab}}

\begin{document}

\title[Skew bracoids]{Skew bracoids}

\author{Isabel Martin-Lyons}
\address{School of Computer Science and Mathematics \\ Keele University \\ Staffordshire \\ ST5 5BG \\ UK}
\email{I.D.Martin-Lyons@Keele.ac.uk}

\author{Paul J. Truman}
\address{School of Computer Science and Mathematics \\ Keele University \\ Staffordshire \\ ST5 5BG \\ UK}
\email{P.J.Truman@Keele.ac.uk}



\begin{abstract}
Skew braces are intensively studied owing to their wide ranging connections and applications. We generalize the definition of a skew brace to give a new algebraic object, which we term a \textit{skew bracoid}. Our construction involves two groups interacting in a manner analogous to the compatibility condition found in the definition of a skew brace. We formulate tools for characterizing and classifying skew bracoids, and study substructures, quotients, homomorphisms, and isomorphisms. As a first application, we prove that finite skew bracoids correspond with Hopf-Galois structures on finite separable extensions of fields, generalizing the existing connection between finite skew braces and Hopf-Galois structures on finite Galois extensions. 
\end{abstract}

\maketitle

\section{Introduction} \label{sec_introduction}

Skew braces are a generalization, due to Guanieri and Vendramin \cite{GV17}, of braces, which were introduced by Rump in \cite{Ru07a}. A (left) \textit{skew brace} is a triple $ (B,\star, \cdot) $ such that $ (B,\star) $ and $ (B,\cdot) $ are groups (sometimes called the \textit{additive} and \textit{multiplicative} groups of the skew brace, respectively) and the following compatibility condition (the (left) \textit{skew brace relation}) is satisfied
\begin{equation} \label{eqn_skew_brace_relation}
a \cdot (b \star c) = (a \cdot b) \star a^{-1} \star (a \cdot c) \mbox{ for all } a,b,c \in B.
\end{equation}
Here $ a^{-1} $ denotes the inverse of $ a $ with respect to $ \star $. A \textit{brace} is a skew brace in which the group $ (B,\star) $ is abelian. These objects were introduced, and continue to be intensively studied, because they yield nondegenerate set theoretic solutions of the \textit{Yang-Baxter equation} \cite[Section 3]{GV17}, \cite[Section 3]{SV18}. They have also been found to have connections with a wide range of other algebraic objects, including near-rings, racks, quandles, pre-lie rings, and Hopf-Galois structures. We give a little more detail on this last construction since we return to it later. A \textit{Hopf-Galois structure} on a finite extension of fields $ L/K $ consists of a $ K $-Hopf algebra $ H $ and a $ K $-linear action of $ H $ on $ L $ satisfying a certain nondegeneracy condition (\cite[(2.7) Definition]{TWE}). A theorem of Greither and Pareigis \cite{GP87} implies that if $ L/K $ is a Galois extension with Galois group $ G $ then the Hopf-Galois structures on $ L/K $ correspond bijectively with certain subgroups $ N $ of $ \Perm(G) $ of the same order as $ G $; the isomorphism class of $ N $ is then called the \textit{type} of the corresponding Hopf-Galois structure. As observed by Bachiller \cite[Remark 2.6]{Ba16} (for $ N $ abelian) and by Smoktunowicz, Vendramin, and Byott \cite[Appendix A]{SV18} more generally, the extension $ L/K $ admits a Hopf-Galois structure of type $ N $ if and only if there exists a skew brace $ (B,\star,\cdot) $ with $ (B,\star) \cong N $ and $ (B,\cdot) \cong G $. This connection has recently been sharpened by Stefanello and Trappeniers \cite{ST22}: writing $ G=(G,\cdot) $, they show that there is a bijection between further binary operations $ \star $ on $ G $ such that $ (G,\star,\cdot) $ is a skew brace and Hopf-Galois structures on $ L/K $. The interactions between skew braces and Hopf-Galois structures have generated an explosion of interest in recent years: see for example \cite{Ch18}, \cite{Ch19b}, \cite{NZ19}, \cite{KT20}, \cite{HAGMT}, \cite{CS21}, \cite{KT22}. In particular, it is known that important structural information about a Hopf-Galois structure, such as the intermediate fields occurring in the image of the so-called \textit{Hopf-Galois correspondence}, is reflected in the corresponding skew brace (\cite{Ch19b}, \cite[Section 5]{KT20}, \cite{ST22}). 

Braces and skew braces have been generalized in a variety of ways. For example: to \textit{semi braces} \cite{CCS17}, to \textit{$q$-braces} \cite{Ru19c}, to \textit{weak braces} \cite{CMMS22}, and to \textit{near braces} \cite{DR23}. 

In this paper we introduce a new generalization of skew braces, which we term \textit{skew bracoids}. Our construction involves two groups $ G $ and $ N $ and a transitive action of the former on the latter that interacts with the multiplication in $ N $ in a manner analogous to the skew brace relation \eqref{eqn_skew_brace_relation}. We give the formal definition in Section \ref{sec_defn_characterization}. We show that a skew brace can be viewed as a skew bracoid, and also that skew bracoids arise naturally as certain quotients of skew braces (Proposition \ref{prop_skew_bracoid_quotient_skew_brace}). We show that several important constructions and identities associated with skew braces (such as \textit{$ \gamma $-functions}) have natural skew bracoid counterparts, and also prove that various approaches to constructing and classifying skew braces generalize to skew bracoids (Theorem \ref{thm_characterizations}). On the other hand, we find that we require a notion \textit{equivalence} of skew bracoids (Definition \ref{defn_equivalent}), which does not have a skew brace counterpart, in order to articulate some of the additional complexities that arise in our theory. 

In Section \ref{sec_substructures_quotients} we study subskew bracoids, along with \textit{left ideals} and \textit{ideals}, and prove that we may form quotient skew bracoids with respect to ideals (Proposition \ref{prop_quotient_skew_bracoid}). In Section \ref{sec_homomorphisms} we define \textit{homomorphisms} and \textit{isomorphisms} of skew bracoids, and establish the results that one would naturally expect: the kernel of a skew bracoid homomorphism is an ideal of the domain, the image is a subskew bracoid of the codomain (Proposition \ref{prop_homomorphism_kernel_image}), and a version of the First Isomorphism Theorem holds (Theorem \ref{thm_FIT}). In addition, we consider how homomorphisms and isomorphisms interact with our notion of equivalence introduced in Section \ref{sec_defn_characterization}. 

Finally, in Section \ref{sec_HGS} we prove there is a correspondence between finite skew bracoids and Hopf-Galois structures on a finite separable extension of fields (Theorem \ref{thm_Hopf_Galois_structures}), generalizing the correspondence between finite skew braces and Hopf-Galois structures on Galois extensions described above. We also show that, as in the Galois case, information about the intermediate fields occurring in the image of the Hopf-Galois correspondence is reflected in the corresponding skew bracoid (Theorem \ref{thm_HG_correspondence}). 

It is natural to ask whether skew bracoids have any connection with set theoretic solutions of the Yang-Baxter equation; we will address this question in a forthcoming paper. 

\textbf{Conventions} All groups, skew braces, skew bracoids, and field extensions are assumed to be finite. The identity element of a group $ A $ is denoted $ e_{A} $. 

\textbf{Acknowledgements} The second author gratefully acknowledges the support of the Engineering and Physical Sciences Research Council, project reference \linebreak EP/W012154/1.

\section{Definitions and characterizations} \label{sec_defn_characterization}

In this section we present the definition of a skew bracoid, give some families of examples, and establish some elementary properties, including results for constructing an characterizing skew bracoids. 

\begin{definition} \label{defn_SBO}
A (left) \textit{skew bracoid} is a $ 5 $-tuple $ (G, \cdot, N, \star, \odot) $ such that $ (G, \cdot) $ and $ (N,\star) $ are groups and $ \odot $ is a transitive action of $ (G, \cdot) $ on $ N $ such that 
\begin{equation} \label{eqn_SBO_relation}
g \odot (\eta \star \mu) = (g \odot \eta) \star (g \odot e_{N})^{-1} \star (g \odot \mu) 
\end{equation} 
for all $ g \in G $ and $ \eta, \mu \in N $. 
\end{definition}

We shall call Equation \eqref{eqn_SBO_relation} the (left) \textit{skew bracoid relation}.

The intuition behind Definition \ref{defn_SBO} is that the group $ (N,\star) $ is analogous to the additive group of a skew brace, and the group $ (G,\cdot) $, together with its transitive action $ \odot $ on $ N $, is analogous to the multiplicative group. The skew bracoid relation \eqref{eqn_SBO_relation} stipulates that the transitive action of $ (G,\cdot) $ on $ N $ should interact with the binary operation $ \star $ on $ N $ in a manner analogous to the skew brace relation \eqref{eqn_skew_brace_relation}. 

To ease notation, we frequently suppress the notation $ \cdot $, and occasionally the notation $ \star $, where there is no risk of confusion. We retain the symbol $ \odot $ for the action of $ G $ on $ N $ in our arguments, but often suppress it when specifying skew bracoids, \textit{viz} $ (G,N) $. 
 
Unlike the theory of skew braces (in which we have two binary operations on the same set), we use the notation $ \,^{-1} $ to denote all inverses, since the group in which the inverse is being taken can be deduced from the context. Thus $ (g \odot e_{N})^{-1} $ denotes the inverse of the $ g \odot e_{N} \in N $ element with respect to $ \star $, whereas $ (g^{-1} \odot e_{N}) $ denotes the action of $ g^{-1} \in G $ on the identity element of $ N $.  

We shall say that a skew bracoid is \textit{finite} to mean that both $ G $ and $ N $ are finite. In this case, by the orbit-stabilizer theorem, the order of $ N $ divides the order of $ G $. All the skew bracoids we consider will be finite.

\begin{example} \label{eg_essentially_skew_brace}
A skew brace $ (B,\star,\cdot) $ can be viewed as a skew bracoid with $ G = (B,\cdot) $, $ N= (B,\star) $, and $ \odot = \cdot $. Where necessary, we express this as $ (B, \star, B, \cdot, \odot)  $ or $ (B,B) $. 

Conversely, if $ (G,N) $ is a skew bracoid in which $ |G|=|N| $ (or, equivalently, $ \Stab_{G}(e_{N}) = \{ e_{G} \} $) then the action of $ G $ on $ N $ is regular (i.e. transitive and fixed-point free), and so we may transport the structure of $ G $ to $ N $ via the rule 
\[ (g \odot e_{N}) \circ (h \odot e_{N}) = (gh) \odot e_{N}. \]
Having done this, the triple $ (N,\star,\circ) $ is a skew brace. Because of this, we shall say that a skew bracoid in which $ |G|=|N| $ is \textit{essentially a skew brace}. 
\end{example}

\begin{example} \label{eg_running_first_principles}
Let $ n \in \mathbb{N} $, let $ d $ be a positive divisor of $ n $, let
\[ G = \langle r, s \mid r^{n} = s^{2} = e_{G}, \; srs^{-1}=r^{-1} \rangle \cong D_{n}, \]
and let $ N = \langle \eta \rangle \cong C_{d} $. Then the rule 
\begin{equation}
r^{i}s^{j} \odot \eta^{k} = \eta^{i+(-1)^{j}k} 
\end{equation}
defines a transitive action of $ G $ on $ N $, and we have
\begin{eqnarray*}
&& (r^{i}s^{j} \odot \eta^{k}) \star (r^{i}s^{j} \odot e_{N})^{-1} \star (r^{i}s^{j} \odot \eta^{\ell}) \\
& = & \eta^{i+(-1)^{j}k} \star \eta^{-i} \star \eta^{i+(-1)^{j}\ell} \\
& = & \eta^{i+(-1)^{j}(k+\ell)} \\
& = & r^{i}s^{j} \odot \eta^{k+\ell}.
\end{eqnarray*}
Therefore the skew bracoid relation \eqref{eqn_SBO_relation} is satisfied, and so $ (G,N,\odot) $ is a skew bracoid. 
\end{example}

Certain skew bracoids arise from a natural quotienting procedure on skew braces. To describe this procedure, we recall first that the \textit{$ \gamma $-function} of a skew brace $ (B,\star,\cdot) $ is the function $ \gamma : B \rightarrow \Perm(B) $ defined by $ \,^{\gamma(b)}a = b^{-1} \star (b \cdot a) $. In fact, we have $ \gamma(b) \in \Aut(B,\star) $ for each $ b \in B $, and the map $ \gamma $ is a homomorphism from $ (B,\cdot) $ to $ \Aut(B,\star) $. A \textit{left ideal} of a skew brace $ (B,\star,\cdot) $ is a subset $ A $ of $ B $ that is a subgroup of $ B $ with respect to $ \star $, a subgroup of $ B $ with respect $ \cdot $, and that satisfies $ \,^{\gamma(B)}A = A $ (the third condition, along with either of the subgroup conditions, implies the other). A left ideal $ A $ of $ (B,\star,\cdot) $ is called a \textit{strong left ideal} if $ (A,\star) $ is normal in $ (B,\star) $, and is called an \textit{ideal} if $ (A,\star) $ is normal in $ (B,\star) $ and $ (A,\cdot) $ is normal in $ (B,\cdot) $. It is well known that quotients of skew braces by ideals are again skew braces. We can generalize this as follows

\begin{proposition} \label{prop_skew_bracoid_quotient_skew_brace}
Let $ (B,\star,\cdot) $ be a skew brace and let $ A $ be a strong left ideal. Then $ (B, \cdot, B/A, \star, \odot) $ is a skew bracoid, where $ (B/A, \star) $ denotes the quotient group of $ (B,\star) $ by $ (A,\star) $ and $ \odot $ denotes left translation of cosets with respect to $ \cdot $. 
\end{proposition}
\begin{proof}
First we note that the cosets of $ A $ in $ (B,\star) $ and $ (B,\cdot) $ coincide: for $ b \in B $ we have
\begin{eqnarray*}
b \star A & = & \{ b \star a \mid a \in A \} \\
& = & \{ b \star \,^{\gamma(b)}a \mid a \in A \} \mbox{ since $ \gamma(b) \in \Aut(A,\star) $ } \\
& = & \{ b \star b^{-1} \star (b \cdot a) \mid a \in A \} \\
& = & \{ b \cdot a \mid a \in A \} \\
& = & b \odot A.
\end{eqnarray*}
Thus it makes sense to consider the set $ B/A $ simultaneously as a quotient group with respect to $ \star $ and as a $ (B,\cdot) $-set via left translation of cosets, and in fact we may write $ b \star A = b \odot A = bA $ without ambiguity. It is clear that the action of $ (B,\cdot) $ on $ B/A $ by left translation is transitive, so it only remains to verify that the skew bracoid relation is satisfied. For $ b,c,d \in B $ we have
\begin{eqnarray*}
b \odot ( cA \star dA ) & = & b \odot ( (c \star d) A ) \\
& = & (b \cdot (c \star d)) A  \\
& = & ( (b \cdot c) \star b^{-1} \star (b \cdot d) ) A \\
& = & (b \cdot c)A \star (bA)^{-1} \star (b \cdot d)A \\
& = & (b \odot cA) \star (b \odot A)^{-1} \star (b \odot dA).
\end{eqnarray*}
(Here $ b^{-1} $ denotes the inverse of $ b $ in $ (B,\star) $, and $ (bA)^{-1} $, $ (b \odot A)^{-1} $ denote the inverses of these cosets in $ (B/A, \star) $.) Thus the skew bracoid relation is satisfied, and so $ (B, \cdot, B/A, \star, \odot) $ is a skew bracoid. 
\end{proof}

\begin{example}\label{eg_running_quotient_of_skew_brace}
The skew bracoid $ (G,N) $ constructed in Example \ref{eg_running_first_principles} can be obtained via the procedure described in Proposition \ref{prop_skew_bracoid_quotient_skew_brace}. We write the group $ G $ as $ (G,\cdot) $, and define a new binary operation on $ G $ by $ r^{i}s^{j} \star r^{k}s^{\ell} = r^{i+k}s^{j+\ell} $; it is then easy to verify that $ (G,\star) $ is a group isomorphic to $ C_{n} \times C_{2} $ and $ (G,\star,\cdot) $ is a skew brace with $ \gamma $-function given by $ \,^{\gamma(r^{i}s^{j})} r^{k}s^{\ell} = r^{(-1)^{j}k}s^{\ell} $. It follows that subgroup $ H = \langle r^{d}, s \rangle $ of $ (G,\star) $ is a strong left ideal of $ (G,\star,\cdot) $. The quotient group $ ( G/H, \star ) $ is cyclic of order $ d $, generated by the coset $ rH $. Now the action of $ (G,\cdot) $ on $ G/H $ by left translation is given by 
\[ r^{i}s^{j} \odot (rH)^{k} = r^{i}s^{j} \odot r^{k}H = r^{i + (-1)^{j}k}H. \]
Thus $ (G/H,\star) $ is isomorphic to the group $ N $ in Example \ref{eg_running_first_principles} as a group and a $ G $-set. We shall give a formal treatment of isomorphism of skew bracoids in Section \ref{sec_homomorphisms}. 
\end{example}

\begin{remark} \label{rmk_every_skew_bracoid_quotient_of_SB}
It is natural to ask whether every skew bracoid can be obtained via the procedure described in Proposition \ref{prop_skew_bracoid_quotient_skew_brace}; at present we cannot answer this question. 
\end{remark}

In \cite{KT20} Koch and the second author introduce the notion of the \textit{opposite} of a skew brace (the same construction is studied independently by Rump in \cite{Ru19c}). Opposite skew braces have applications to the Yang-Baxter equation \cite[Section 4]{KT20},  \cite{KST20}, \cite{CJVAV22}, to enumeration and classification problems \cite{CCDC20}, \cite{Ca20}, and in Hopf-Galois theory \cite{KT22}, \cite{CS21}. In particular, they play an important role in the Stefanello-Trappeniers approach to the correspondence between skew braces and Hopf-Galois structures on Galois fields extensions \cite{ST22}. The concept of opposites extends naturally to skew bracoids:

\begin{proposition}
Let $ (G,\cdot, N,\star,\odot) $ be a skew bracoid, and let $ (N, \star^{op}) $ be the opposite group to $ N $. Then $ (G,\cdot, N,\star^{op}, \odot) $ is a skew bracoid. 
\end{proposition}
\begin{proof}
It is clear that $ \odot $ gives a transitive action of $ G $ on $ N $. We must show that the skew bracoid relation holds for $ (G,\cdot, N,\star^{op}, \odot) $. Let $ g \in G $ and $ \eta, \mu \in N $. Then we have 
\begin{eqnarray*}
g \odot (\eta \star^{op} \mu) & = & g \odot (\mu \star \eta) \\
& = & (g \odot \mu) \star (g \odot e_{N})^{-1} \star (g \odot \eta) \\
& = & (g \odot \eta) \star^{op} (g \odot e_{N})^{-1} \star^{op} (g \odot \mu). 
\end{eqnarray*}
Hence $ (G,\cdot, N,\star^{op}, \odot) $ is a skew bracoid. 
\end{proof} 

Next we turn to the question of characterizing skew bracoids. We generalize results of Guanieri and Vendramin (\cite[Proposition 1.11 and Theorem 4.2]{GV17}). They prove that, given a group $ (B,\star) $, the following are equivalent:
\begin{itemize}
\item a binary operation $ \cdot $ on $ B $ such that $ (B,\star,\cdot) $ is a skew brace; 
\item a regular subgroup of the \textit{holomorph} $ \Hol_{\star}(B) $ (this is the normalizer of $ \lambda_{\star}(B) $ in $ \Perm(B) $, and is equal to the semidirect product of $ \lambda_{\star}(B) $ and $ \Aut_{\star}(B) $);
\item a group $ G $ acting on $ (B,\star) $ by automorphisms together with a bijective $ 1 $-cocycle $ \pi : G \rightarrow B $. 
\end{itemize}

In the case of skew bracoids, we have

\begin{theorem}  \label{thm_characterizations}
Let $ (G,\cdot), (N,\star) $ be groups. The following are equivalent:
\begin{enumerate}[ref=\roman*]
\item A transitive action $ \odot $ of $ G $ on $ N $ such that $ (G,\cdot,N,\star,\odot) $ is a skew bracoid; \label{enum_thm_characterization_1}
\item a transitive subgroup $ A $ of $ \Hol(N) $ isomorphic to a quotient of $ G $; \label{enum_thm_characterization_2}
\item a homomorphism $ \gamma : G \rightarrow \Aut(N) $ and a surjective $ 1 $-cocycle $ \pi : G \twoheadrightarrow N $. \label{enum_thm_characterization_3}
\end{enumerate}
\end{theorem}
\begin{proof}
First suppose that \eqref{enum_thm_characterization_1} holds. Let $ \lambda_{\odot} : G \rightarrow \Perm(N) $ be the permutation representation of the action $ \odot $. Then $ \lambda_{\odot}(G) $ is a transitive subgroup of $ \Perm(N) $ isomorphic to a quotient of $ G $. To show that $ \lambda_{\odot}(G) \subseteq \Hol(N) $ we consider the functions $ \gamma(g) \in \Perm(N) $ defined by
\[ \gamma(g) = \lambda_{\star}(g \odot e_{N})^{-1} \lambda_{\odot}(g), \]
so that
\[ \,^{\gamma(g)} \eta = (g \odot e_{N})^{-1} \star (g \odot \eta) \mbox{ for all } g \in G \mbox{ and } \eta \in N. \]
We claim that $ \gamma(g) \in \Aut(N) $ for all $ g \in G $; if this holds then we quickly obtain
\[ \lambda_{\odot}(g) = \lambda_{\star}(g \odot e_{N}) \gamma(g) \in \lambda_{\star}(N)\Aut(N) = \Hol(N) \mbox{ for all } g \in G. \]
To prove the claim, let $ g \in G $ and $ \eta, \mu \in N $; then we have
\begin{eqnarray*}
\,^{\gamma(g)} (\eta \star \mu) & = & (g \odot e_{N})^{-1} \star (g \odot (\eta \star \mu)) \\
& = & (g \odot e_{N})^{-1} \star (g \odot \eta) \star (g \odot e_{N})^{-1} \star (g \odot \mu) \mbox{ by \eqref{eqn_SBO_relation}} \\
& = & \left( \,^{\gamma(g)} \eta \right) \star \left( \,^{\gamma(g)} \mu \right).
\end{eqnarray*}
Thus $ \gamma(g) $ is a homomorphism. To show that it is bijective, suppose that $ \eta \in \ker(\gamma(g)) $. Then
\begin{eqnarray*}
&& (g \odot e_{N})^{-1} \star (g \odot \eta) = e_{N} \\
& \Rightarrow &  (g \odot \eta) = (g \odot e_{N}) \\
& \Rightarrow & \eta = e_{N}.
\end{eqnarray*}
Hence $ \ker(\gamma(g)) $ is trivial and so, since $ N $ is finite, $ \gamma(g) $ is bijective. Therefore $ \gamma(g) \in \Aut(N) $ as claimed, so $ A=\lambda_{\odot}(G) $ is a transitive subgroup of $ \Hol(N) $ isomorphic to a quotient of $ G $, and so \eqref{enum_thm_characterization_2} holds. 

Next suppose that \eqref{enum_thm_characterization_2} holds. Let $ \delta : G \rightarrow A $ be a surjective homomorphism, and for each $ g \in G $ write $ \delta(g) = \lambda_{\star}(\pi(g))\gamma(g) $ with $ \pi(g) \in N $ and $ \gamma(g) \in \Aut(N) $. Then $ \gamma : G \rightarrow \Aut(N) $ is the composition the homomorphism $ \delta $ with the projection onto the automorphism component; it is therefore a homomorphism, and so yields and action of $ G $ on $ N $ by automorphisms. Since $ A $ is a transitive subgroup of $ \Hol(N) $, we have
\[ N = A[e_{N}] = \{ \pi(g) \star \,^{\gamma(g)} e_{N} \mid g \in G \} = \{ \pi(g) \mid g \in G \}, \]
so the map $ \pi : G \rightarrow N $ is surjective. Finally, the assumption that $ \delta $ is a homomorphism implies that for $ g,h \in G $ we have 
\begin{eqnarray*}
&& \delta(gh)[e_{N}] = \delta(g)(\delta(h)[e_{N}]) \\
& \Rightarrow & \pi(gh) \,^{\gamma(gh)}e_{N} = \pi(g) \star \,^{\gamma(g)} [ \pi(h) \star \,^{\gamma(h)}e_{N} ] \\
& \Rightarrow & \pi(gh) =  \pi(g) \star \,^{\gamma(g)} \pi(h). 
\end{eqnarray*}
Thus $ \pi $ is a surjective $ 1 $-cocycle for the action of $ G $ on $ N $ by automorphisms via $ \gamma $, and so \eqref{enum_thm_characterization_3} holds. 

Finally, suppose that \eqref{enum_thm_characterization_3} holds. For $ g \in G $ and $ \eta \in N $ define $ g \odot \eta \in N $
\[ g \odot \eta = \pi(g) \star \,^{\gamma(g)} \eta. \]
The assumption that $ \pi $ is a $ 1 $-cocycle for the action of $ G $ on $ N $ by automorphisms via $ \gamma $ implies $ \odot $ defines an action of $ G $ on $ N $: it is easy to see that $ \pi(e_{G})=e_{N} $, so $ e_{G} \odot \eta = \eta $ for all $ \eta \in N $, and for $ g, h \in G $ and $ \eta \in N $ we have
\begin{eqnarray*}
g \odot ( h \odot \eta ) & = & \pi(g) \star \,^{\gamma(g)} [ \pi(h) \star \,^{\gamma(h)} \eta ] \\
& = & \pi(g) \star \,^{\gamma(g)} \pi(h) \star \,^{\gamma(g)\gamma(h)} \eta  \\
& = & \pi(gh) \star \,^{\gamma(gh)} \eta  \\
& = & (gh) \odot \eta.
\end{eqnarray*}
Furthermore, the assumption that $ \pi $ is surjective implies that this action is transitive: we have
\[ G \odot e_{N} = \{ \pi(g) \star \,^{\gamma(g)} e_{N} \mid g \in G \} = \{ \pi(g) \mid g \in G \} = N. \]
Finally, for $ g \in G $ and $ \eta, \mu \in N $ we have
\begin{eqnarray*}
g \odot (\eta \star \mu) & = & \pi(g) \star \,^{\gamma(g)}[\eta \star \mu] \\
& = & \pi(g) \star \,^{\gamma(g)}\eta \star \,^{\gamma(g)}\mu \\
& = & \pi(g) \star \,^{\gamma(g)}\eta \star \pi(g)^{-1} \star \pi(g) \star \,^{\gamma(g)}\mu \\
& = & (g \odot \eta) \star (g \odot e_{N})^{-1} \star (g \odot \mu), 
\end{eqnarray*}
so \eqref{eqn_SBO_relation} holds. 
\end{proof}

\begin{example}\label{eg_running_characterization}
Recall the skew bracoid  $ (G,N) $ constructed in Example \ref{eg_running_first_principles}. Beginning with the action of $ G $ on $ N $ given in that example, we have 
\[ \lambda_{\odot}(r^{i}s^{j})[\eta^{k}] = \eta^{i+(-1)^{j}k} = \lambda_{\star}(\eta^{i}) \iota^{j}[\eta^{k}], \]
where $ \iota \in \Aut(N) $ is the automorphism given by inversion. Thus $ \lambda_{\odot}(G) = \lambda_{\star}(N)\langle \iota \rangle \subset \Hol(N) $, and $ \gamma(r^{i}s^{j}) = \iota^{j} $. 

Secondly, (for example) the function $ \delta : G \rightarrow \lambda_{\star}(N)\langle \iota \rangle $ defined by $ \delta(r^{i}s^{j}) = \lambda_{\star}(\eta^{i}) \iota^{j} $ is a surjective homomorphism; the projection onto the $ \Aut(N) $ component is $ \gamma(r^{i}s^{j}) = \iota^{j} $, and the corresponding $ 1 $-cocycle is given by $ \pi(r^{i}s^{j}) = \eta^{i} $. Explicitly, we have
\begin{eqnarray*}
\pi( r^{i}s^{j} r^{k}s^{\ell} ) & = & \pi( r^{i+(-1)^{j}k} ) \\
& = & \eta^{i+(-1)^{j}k} \\
& = & \eta^{i} \star \iota^{j} (\eta^{k}) \\
& = & \pi(r^{i}s^{j}) \,^{\gamma(r^{i}s^{j})} \pi(r^{k}s^{\ell}). 
\end{eqnarray*}

Finally, from $ \gamma $ and $ \pi $ we may define a transitive action of $ G $ on $ N $ by 
\begin{eqnarray*}
r^{i}s^{j} \odot \eta^{k} & = & \pi(r^{i}s^{j}) \star \,^{\gamma(r^{i}s^{j})} \eta^{k} \\
& = & \eta^{i} \eta^{(-1)^{j}k} \\
& = & \eta^{i+(-1)^{j}k}.
\end{eqnarray*}
We see that we recover the original action of $ G $ on $ N $, as in Example \ref{eg_running_first_principles}. 
\end{example}

From the proof of Theorem \ref{thm_characterizations} we extract the following, which will be used frequently in what follows. 

\begin{definition} \label{defn_gamma_function}
Let $ (G,N) $ be a skew bracoid. The homomorphism $ \gamma : G \rightarrow \Aut(N) $ defined by
\[ \,^{\gamma(g)} \eta = (g \odot e_{N})^{-1} \star (g \odot \eta) \mbox{ for all } g \in G \mbox{ and } \eta \in N \]
is called the \textit{ $\gamma $-function} of the skew bracoid. 
\end{definition}

\begin{example} \label{eg_quotient_gamma_function}
When a skew brace is viewed as a skew bracoid, the $ \gamma $-function of the skew bracoid coincides with that of the skew brace. 

If $ (B, B/A) $ is a skew bracoid arising from a skew brace $ (B,\star,\cdot) $ and a strong left ideal $ A $, as described in Proposition \ref{prop_skew_bracoid_quotient_skew_brace}, then its $ \gamma $-function is given by 
\begin{eqnarray*}
\,^{\gamma(b)} (cA) & = & (b \odot e_{B}A)^{-1} \star (b \odot cA) \\
& = & (bA)^{-1} \star ((b \cdot c)A) \\
& = & (b^{-1} \star (b \cdot c)) A \\
& = & (\,^{\gamma(b)}c)A. 
\end{eqnarray*}
\end{example}

\begin{example} \label{eg_running_gamma_function}
The skew bracoid  $ (G,N) $ constructed in Example \ref{eg_running_first_principles} has $ \gamma $-function given by 
\begin{eqnarray*}
\,^{\gamma(r^{i}s^{j})} \eta^{k} & = & (r^{i}s^{s} \odot e_{N})^{-1} \star (r^{i}s^{j} \odot \eta^{k}) \\
& = & \eta^{-i} \star \eta^{i+(-1)^{j}k} \\
& = & \eta^{(-1)^{j}k}. 
\end{eqnarray*}
\end{example}

We record a useful consequence of the fact that the $ \gamma $-function of a skew bracoid $ (G,N) $ has values in $ \Aut(N) $.

\begin{proposition} \label{prop_useful_identites} 
Let $ (G,N) $ be a skew bracoid. Then for all $ g \in G $ and $ \eta \in N $ we have
\[ (g \odot e_{N})^{-1} \star (g \odot \eta^{-1}) \star (g \odot e_{N})^{-1} = (g \odot \eta)^{-1}. \]
\end{proposition}
\begin{proof}
Let $ g \in G $ and $ \eta \in N $. Since $ \gamma(g) \in \Aut(N) $ we have 
\begin{eqnarray*}
&& \,^{\gamma(g)}(\eta^{-1}) = \left( \,^{\gamma(g)}\eta \right)^{-1} \\
& \Rightarrow & (g \odot e_{N})^{-1} \star (g \odot \eta^{-1}) = ( (g \odot e_{N})^{-1} \star (g \odot \eta) )^{-1} \\
& \Rightarrow & (g \odot e_{N})^{-1} \star (g \odot \eta^{-1}) =  (g \odot \eta)^{-1} \star (g \odot e_{N}),
\end{eqnarray*}
from which the identity follows immediately. 
\end{proof}

If $ (B,\star,\cdot) $ is a skew brace then the image of the left regular representation $ \lambda_{\bullet} : (B,\cdot) \rightarrow \Perm(B) $ is a regular subgroup of $ \Hol_{\star}(B) $ (see \cite[Theorem 4.2]{GV17}). The stabilizer of each element of $ B $ with respect to $ \cdot $ is trivial and so, \textit{a fortiori}, $ \ker(\lambda_{\bullet}) $ is trivial. In a skew bracoid $ (G,\cdot,N,\star,\odot) $ the analogue of the first of these is almost always false: by the orbit-stabilizer theorem we have $ |G| = |\Stab_{G}(\eta)||N| $ for each $ \eta \in N $ so, unless $ (G,N) $ is essentially a skew brace, each $ \Stab_{G}(\eta) $ is nontrivial (of course, these stabilizers are mutually conjugate in $ G $). However, it may still happen that $ \ker(\lambda_{\odot}) $ (the intersection of these stabilizers) is trivial, and so we make the following definition:

\begin{definition} \label{defn_reduced}
We say that a skew bracoid $ (G,N) $ is \textit{reduced} to mean that $ \lambda_{\odot} $ is injective or, equivalently, that the action of $ G $ on $ N $ via $ \odot $ is faithful.
\end{definition}

\begin{example}
A skew brace, viewed as a skew bracoid, is reduced. 

If $ (B, B/A) $ is a skew bracoid arising from a skew brace $ (B,\star,\cdot) $ and a strong left ideal $ A $, as described in Proposition \ref{prop_skew_bracoid_quotient_skew_brace}, then the stabilizer of the identity coset $ e_{B}A $ is precisely $ A $, and more generally the stabilizer of a coset $ b \odot A $ is $ b \cdot A \cdot b^{-1} $. It follows that $ \ker(\lambda_{\odot}) = \bigcap_{b \in B} b \cdot A \cdot b^{-1} $, the normal core of $ A $ in $ (B,\cdot) $. Therefore $ (B, B/A) $ is reduced if and only if $ (A, \cdot) $ is core-free. 
\end{example}

\begin{example} \label{eg_running_kernel_of_lambda}
Recall the skew bracoid $ (G,N) $ constructed in Example \ref{eg_running_first_principles}. Since $ N $ is cyclic of order $ d $ and $ r^{i}s^{j} \odot \eta^{k} = \eta^{i+(-1)^{j}k} $, we have $ \ker(\lambda_{\odot}) = \langle r^{d} \rangle $, so $ (G,N) $ is reduced if and only if $ d=n $. 
\end{example}

If $ (G,N) $ is a skew bracoid that is not reduced then the elements of $ \ker(\lambda_{\odot}) $ have no effect on the structure of $ N $ as a $ G $-set, and so are, in some sense, superfluous. It is therefore tempting to revise our definition of a skew bracoid and insist that $ \odot $ should be a \textit{faithful} action of $ G $ on $ N $. However, we shall see in Section \ref{sec_substructures_quotients} that subskew bracoids and quotient skew bracoids of a reduced skew bracoid may not themselves be reduced; therefore the additional flexibility in Definition \ref{defn_SBO} is necessary. On the other hand, given a group $ N $, we clearly require a method for relating skew bracoids $ (G,N) $ and $ (G',N) $ in which the actions of $ G $ and $ G' $ on $ N $ are ``essentially the same". This motivates the following:

\begin{definition} \label{defn_equivalent}
Two skew bracoids $ (G,N) $ and $ (G',N') $ are called \textit{equivalent} if $ (N,\star)=(N',\star') $ and $ \lambda_{\odot}(G) = \lambda_{\odot'}(G') \subseteq \Hol(N) $. This is denoted $ (G,N) \sim (G,N') $.
\end{definition}

The intuition behind this definition is the result from the theory of skew braces that if $ (B,\star) $ is a group and $ \cdot, \cdot' $ are two further binary operations on $ B $ such that $ (B,\star,\cdot) $ and $ (B,\star,\cdot') $ are skew braces, then we have $ (B,\star,\cdot) = (B,\star,\cdot') $ if and only if $ \lambda_{\bullet}(B) = \lambda_{\bullet'}(B) $ inside $ \Hol_{\star}(B) $. 

It is clear that equivalence of skew bracoids is an equivalence relation. 

\begin{proposition} \label{prop_reduced_form}
Let $ (G,N) $ be a skew bracoid, let $ K = \ker(\lambda_{\odot}) $, and let $ \bar{G} = G / K $. Then the group $ \bar{G} $ acts on $ N $ via $ (gK) \odot \eta = g \odot \eta $, and $ (\bar{G}, N) $ is a reduced skew bracoid, called the \textit{reduced form} of $ (G,N) $. 
\end{proposition}
\begin{proof}
If $ gK = hK $ for some $ g,h \in G $ then $ g^{-1}h \in K $, so $ g^{-1}h \odot \eta = \eta $ for all $ \eta \in N $, and so $ g \odot \eta = h \odot \eta $ for all $ \eta \in N $; therefore the rule in the statement of the proposition is well defined. It follows quickly from the fact that $ (G,N) $ is a skew bracoid that this rule defines a transitive action of $ \bar{G} $ on $ N $ and that the skew bracoid relation is satisfied. Therefore $ (\bar{G},N) $ is a skew bracoid. To show that it is reduced, suppose that $ gK \in \bar{G} $ is such that $ (gK) \odot \eta = \eta $ for all $ \eta \in N $. Then $ g \odot \eta = \eta $ for all $ \eta \in N $, so $ g \in K $, so $ gK = eK $. Therefore $ (\bar{G},N) $ is a reduced skew bracoid. 
\end{proof}

From the definition of equivalence we obtain immediately

\begin{corollary}
A skew bracoid is equivalent to its reduced form. 
\end{corollary}

\begin{example} \label{eg_running_reduced_form}
Recall from Example \ref{eg_running_kernel_of_lambda} that for the skew bracoid $ (G,N) $ constructed in Example \ref{eg_running_first_principles} we have $ \ker(\lambda_{\odot}) = \langle r^{d} \rangle $. Therefore for this skew bracoid we have $ \bar{G} = G / \langle r^{d} \rangle \cong D_{n/d} $.  
\end{example}

Combining Proposition \ref{prop_reduced_form} with Theorem \ref{thm_characterizations} we find

\begin{corollary}
Given a group $ N $, there is a bijection between transitive subgroups of $ \Hol(N) $ and equivalence classes of skew bracoids $ (G,N) $. 
\end{corollary}

From the definition of the $ \gamma $-function of a skew bracoid (Definition \ref{defn_gamma_function}) we see that if $ (G,N) $ is a skew bracoid then $ K = \ker(\lambda_{\odot}) \subseteq \ker(\gamma) $. Therefore $ \gamma $ factors through $ G / \ker(\lambda_{\odot}) $, and the $ \gamma $-function of $ (\bar{G},N) $ is given by $ \,^{\bar{\gamma}(gK)} \eta = \,^{\gamma(g)} \eta $ for all $ g \in G $ and $ \eta \in N $. 

\section{Substructures and quotients} \label{sec_substructures_quotients}

The natural next step in our development of the theory of skew bracoids is to study substructures and quotients. Recall that if $ (B,\star,\cdot) $ is a skew brace then a subset $ A $ of $ B $ is called a \textit{subskew brace} if it is closed under both $ \star $ and $ \cdot $, is called a \textit{left ideal} if it is a subskew brace satisfying $ \,^{\gamma(B)}A=A $, is called an \textit{strong left ideal} if it is a left ideal and is normal in $ (B,\star) $, and is called an \textit{ideal} if it is a strong left ideal that if also normal in $ (B,\cdot) $. If $ A $ is an ideal of $ (B,\star,\cdot) $ then $ (B/A, \star, \cdot) $ is a skew brace \cite[Lemma 2.3]{GV17}; we have already seen in Proposition \ref{prop_skew_bracoid_quotient_skew_brace} that if $ A $ is a strong left ideal of $ (B,\star,\cdot) $ then $ (B,\cdot,B/A,\star,\odot) $ is a skew bracoid.  

To ease notation in this section, we frequently suppress the notation for the binary operations $ \cdot $ and $ \star $.

\begin{definition}
A \textit{subskew bracoid} of a skew bracoid $ (G,N) $ consists of a subgroup $ H $ of $ G $ and a subgroup $ M $ of $ N $ such that $ (H,M) $ is a skew bracoid. 
\end{definition}

\begin{example}
When a skew brace is viewed as a skew bracoid, the subskew bracoids are precisely the subskew braces.

If $ (B, B/A) $ is a skew bracoid arising from a skew brace $ (B,\star,\cdot) $ and a strong left ideal $ A $, as described in Proposition \ref{prop_skew_bracoid_quotient_skew_brace} and $ C $ is a subskew brace of $ (B,\star,\cdot) $ that contains $ A $, then $ (C, C/A) $ is a subskew bracoid of $ (B,B/A) $. 
\end{example}

\begin{example} \label{example_running_sub_skew_bracoid}
Recall the skew bracoid $ (G,N) $ constructed in Example \ref{eg_running_first_principles}. Let $ f $ be a positive divisor of $ d $, let $ M = \langle \eta^{f} \rangle$, and let $ H = \langle r^{f}, s \rangle $. Then $ (H,M) $ is a subskew bracoid of $ (G,N) $. 
\end{example}

A subskew bracoid of a reduced skew bracoid need not be reduced:

\begin{example} \label{eg_running_sub_bracoid_not_reduced}
Consider the skew bracoid $ (G,N) $ constructed in Example \ref{eg_running_first_principles} in the particular case in which $ n=d=4 $. By Example \ref{eg_running_kernel_of_lambda}, $ (G,N) $ is reduced. Let $ M = \langle \eta^{2} \rangle $ and $ H = \langle r^{2}, s \rangle $; then $ (H,M) $ is a subskew bracoid of $ (G,N) $. But $ (H,N) $ is not reduced: we have $ s \odot \eta^{2k} = \eta^{-2k} = \eta^{2k} $ for all $ k $, so $ \langle s \rangle $ is contained in (and in fact equals) the kernel of $ \lambda_{\odot} $ in $ H $. 
\end{example}

\begin{definition} \label{defn_left_ideal_and_ideal}
A \textit{left ideal} of a skew bracoid $ (G,N) $ is a subgroup $ M $ of $ N $ such that $ \,^{\gamma(G)} M = M $. An \textit{ideal} of $ (G,N) $ is a left ideal which is normal in $ N $. 
\end{definition}

\begin{example} \label{eg_quotient_ideals}
When a skew brace is viewed as a skew bracoid, the left ideals of the skew bracoid are precisely the left ideals of the skew brace. However, the ideals of the skew bracoid are the strong left ideals of the skew brace. 

If $ (B, B/A) $ is a skew bracoid arising from a skew brace $ (B,\star,\cdot) $ and a strong left ideal $ A $, as described in Proposition \ref{prop_skew_bracoid_quotient_skew_brace}, then we have seen  in Example \ref{eg_quotient_gamma_function} that the $ \gamma $-function of $ (B, B/A) $ is given by $ \,^{\gamma(b)}(cA) = (\,^{\gamma(b)}c)A $. Therefore the left ideals (resp. ideals) of $ (B, B/A) $ are the sets $ C/A $ where $ C $ is a left ideal (resp. strong left ideal) of $ B $ that contains $ A $. 
\end{example}

\begin{example} \label{eg_running_ideals}
Recall the skew bracoid $ (G,N) $ constructed in Example \ref{eg_running_first_principles}. Let $ f $ be a positive divisor of $ d $, and let $ M = \langle \eta^{f} \rangle $. By Example \ref{eg_running_gamma_function} the $ \gamma $-function of $ (G,N) $ is given by $ \,^{\gamma(r^{i}s^{j})} \eta^{k} = \eta^{(-1)^{j}k} $; hence $ M $ is closed under $ \gamma(G) $ and is therefore a left ideal of $ (G,N) $. In fact, since $ N $ is abelian, $ M $ is an ideal of $ (G,N) $. 
\end{example}

In the theory of skew braces we have the useful result that if $ (B,\star,\cdot) $ is a skew brace and $ A $ is a subgroup of $ (B,\star) $ such that $ \,^{\gamma(B)}A=A $ then $ A $ is necessarily a subgroup of $ (B,\cdot) $, and therefore a left ideal of $ (B,\star,\cdot) $. We have a skew bracoid analogue of this result:

\begin{proposition} \label{prop_left_ideal_sub_SBO}
Let $ (G,N) $ be a skew bracoid and let $ M $ be a left ideal of $ (G,N) $. Define
\[ G_{M} = \{ g \in G \mid g \odot \mu \in M \mbox{ for all } \mu \in M \}. \]
Then $ (G_{M}, M) $ is a subskew bracoid of $ (G,N) $. 
\end{proposition}
\begin{proof}
It is clear that that $ G_{M} $ acts on $ M $ and that the skew bracoid relation is satisfied. We need to show that the action of $ G_{M} $ on $ M $ is transitive. Define
\[ G_{M}' = \{ g \in G \mid g \odot e_{N} \in M \}. \]
Since $ M $ is a left ideal of $ (G,N) $, given $ g \in G $ we have  
\[ \,^{\gamma(g)}\mu = (g \odot e_{N})^{-1}(g \odot \mu) \in M \mbox{ for all } \mu \in M, \]
and so $ g \odot \mu \in M $ for all $ \mu \in M $ if and only if $ g \odot e_{N} \in M $. Therefore $ G_{M} = G_{M}' $, and now it is clear that the action of $ G_{M} $ on $ M $ is transitive. 
\end{proof}

Note that the subgroup $ G_{M} $ associated to the left ideal $ M $ in Proposition \ref{prop_left_ideal_sub_SBO} is, by construction, the largest subgroup of $ G $ that acts on $ M $. 

\begin{example} \label{eg_quotient_ideals_G_M}
When a skew brace $ (B,\star,\cdot) $ is viewed as a skew bracoid, the subgroup $ B_{A} $ of $ (B,\cdot) $ associated to a left ideal $ A $ is precisely $ A $ itself. 

If $ (B, B/A) $ is a skew bracoid arising from a skew brace $ (B,\star,\cdot) $ and a strong left ideal $ A $, as described in Proposition \ref{prop_skew_bracoid_quotient_skew_brace}, then we have seen in Example \ref{eg_quotient_ideals} that the left ideals of $ (B, B/A) $ are the sets $ C/A $ where $ C $ is a left ideal of $ B $ that contains $ A $. The subgroup $ B_{C/A} $ of $ (B,\cdot) $ associated to such a left ideal  is $ C $.  
\end{example}

\begin{example} \label{eg_running_ideals_G_M}
Recall the skew bracoid $ (G,N) $ constructed in Example \ref{eg_running_first_principles}. Let $ e $ be a positive divisor of $ d $, and let $ M = \langle \eta^{e} \rangle $. By Example \ref{eg_running_ideals} this is an ideal of $ (G,N) $. We find that $ G_{M} = \langle r^{e}, s \rangle $. 
\end{example}

If $ (B,\star,\cdot) $ is a skew brace and $ A $ is a subgroup of $ (B,\cdot) $ such that $ \,^{\gamma(B)}A=A $ then $ A $ is necessarily a subgroup of $ (B,\star) $, and therefore a left ideal of $ (B,\star,\cdot) $. We also have a skew bracoid analogue of this result:

\begin{proposition}
Let $ (G,N) $ be a skew bracoid, let $ H $ be a subgroup of $ G $ and let 
\[ M_{H} = \{ h \odot e_{N} \mid h \in H \}. \]
Suppose that $ \,^{\gamma(G)} M_{H} = M_{H} $. Then $ M_{H} $ is a left ideal of $ (G,N) $.  
\end{proposition}
\begin{proof}
It is clear that $ H $ acts transitively on $ M_{H} $ and that the skew bracoid relation is satisfied. We need to show that $ M_{H} $ is a subgroup of $ N $. Let $ \eta, \mu \in M_{H} $ and write $ \eta = h \odot e_{N} $ and $ \mu = k \odot e_{N} $ with $ h,k \in H $. By the definition of $ M_{H} $, the element $ h^{-1} \odot \mu = (h^{-1}k) \odot e_{N} $ lies in $ M_{H} $, and since $ \,^{\gamma(g)} M_{H} = M_{H} $, we have $ \,^{\gamma(h)}(h^{-1} \odot \mu) \in M_{H} $. But we have
\begin{eqnarray*}
\,^{\gamma(h)}(h^{-1} \odot \mu) & = & (h \odot e_{N})^{-1} ( h \odot (h^{-1} \odot \mu)) \\ 
& = & \eta^{-1} \mu.
\end{eqnarray*}
Thus $ \eta^{-1}\mu \in M_{H} $, and so $ M_{H} $ is a subgroup of $ N $. 
\end{proof}

If $ (G,N) $ is a skew bracoid and $ M $ is a subset of $ N $ then the structural properties of $ M $ are unaffected by reduction:

\begin{proposition} \label{prop_left_ideals_and_reduced_forms}
Let $ (G,N) $ be a skew bracoid. A subset $ M \subseteq N $ is a subskew bracoid, (resp. left ideal, ideal) of $ (G,N) $ if and only if it is a subskew bracoid (resp. left ideal, ideal) of $ (\bar{G},N) $.
\end{proposition}
\begin{proof}
Recall from Proposition \ref{prop_reduced_form} that $ \bar{G} = G / K $ (with $ K=\ker(\lambda_{\odot}) $) and that the transitive action of $ \bar{G} $ on $ N $ is given by $ (gK) \odot \eta = g \odot \eta $. If $ (H,M) $ is a subskew bracoid of $ (G,N) $ then without loss of generality we may assume that $ \ker(\lambda_{\odot}) \subseteq H $; if not, we may replace $ H $ with the subgroup $ H\ker(\lambda_{\odot}) $, which still acts transitively on $ M $. Then $ \bar{H} $ acts transitively on $ M $, so $ (\bar{H},M) $ is a subskew bracoid of $ (\bar{G},N) $. Conversely, a subskew bracoid of $ (\bar{G},N) $ has the form $ (\bar{H},M) $ for some subgroup $ H $ of $ G $ containing $ K $, and we see that $ (H,M) $ is a subskew bracoid of $ (G,N) $. 

The $ \gamma $-function of $ (\bar{G},N) $ is given by $ \,^{\bar{\gamma}(gK)}\eta = \,^{\gamma(g)}\eta $; hence we have $ \,^{\gamma(G)}M = M $ if and only if $ \,^{\bar{\gamma}(\bar{G})} M = M $, and so $ M $ is a left ideal of $ (G,N) $ if and only if it is a left ideal of $ (\bar{G},M) $. 

Finally, a left ideal $ M $ of $ (G,N) $ is an ideal if and only if it is a normal subgroup of $ N $; since this condition is unrelated to the action of $ G $ or $ \bar{G} $ this occurs if and only if $ M $ is an ideal of $ (\bar{G},N) $.
\end{proof}

Next we prove that the quotient of a skew bracoid by an ideal is a skew bracoid.  

\begin{proposition} \label{prop_quotient_skew_bracoid}
Let $ (G,N) $ be a skew bracoid and let $ M $ be an ideal of $ (G,N) $. Then $ (G,\cdot) $ acts on the quotient group $ N/M $ via $ g \odot (\eta M) = (g \odot \eta)M $, and $ (G,N/M) $ is a skew bracoid.
\end{proposition}
\begin{proof}
First we show that formula for the action of $ G $ on $ N/M $ is well defined. Let $ g \in G $ and $ \eta, \kappa \in N $, and suppose that $ \eta M = \kappa M $. Then $ \eta^{-1}\kappa \in M $, and since $ M $ is an ideal of $ (G,N) $ we have $ \,^{\gamma(g)}(\eta^{-1}\kappa) \in M $. That is:
\[ (g \odot e_{N})^{-1} (g \odot (\eta^{-1}\kappa)) \in M. \]
Applying the skew bracoid relation \eqref{eqn_SBO_relation} we have
\begin{equation} \label{eqn_quotient_SBO_1}
(g \odot e_{N})^{-1} (g \odot \eta^{-1}) (g \odot e_{N})^{-1} (g \odot \kappa) \in M.
\end{equation}
Now applying Proposition \ref{prop_useful_identites}  to the first three terms of \eqref{eqn_quotient_SBO_1} we obtain
\[  (g \odot \eta)^{-1} (g \odot \kappa) \in M, \]
and so
\[ (g \odot \eta)M = (g \odot \kappa)M. \]
Therefore the formula for the action of $ G $ on $ N/M $ is well defined. Since $ (G,N) $ is a skew bracoid, it follows quickly that the action of $ G $ on $ N/M $ is transitive and that the skew bracoid relation is satisfied. 
\end{proof}

\begin{example}
When a skew brace is viewed as a skew bracoid, the ideals of the skew bracoid are the strong left ideals of the skew brace (see Example \ref{eg_quotient_ideals}). Using such an ideal to form a quotient skew bracoid is precisely the process described in Proposition \ref{prop_skew_bracoid_quotient_skew_brace}. 

If $ (B, B/A) $ is a skew bracoid arising from a skew brace $ (B,\star,\cdot) $ and a strong left ideal $ A $, as described in Proposition \ref{prop_skew_bracoid_quotient_skew_brace}, then we have seen in Example \ref{eg_quotient_ideals} that the ideals of $ (B, B/A) $ are the sets $ C/A $ where $ C $ is a strong left ideal of $ B $ that contains $ A $. The corresponding quotient skew bracoid is $ (B, (B/A)/(C/A)) $. 
\end{example}

It is easily shown (similarly to Example \ref{eg_quotient_ideals}) that the $ \gamma $-function of a quotient skew bracoid $ (G,N/M) $ is given by the composition of the $ \gamma $ function of $ (G,N) $ with the natural projection $ N \twoheadrightarrow N/M $, so that $ \,^{\gamma(g)}(\eta M) = (^{\gamma(g)}\eta) M $ for all $ \eta \in N $ and $ g \in G $. We have $ \Stab_{G}(e_{N}M) = G_{M} $ (see Proposition \ref{prop_left_ideal_sub_SBO}); it follows quickly from this that we have $ \Stab_{G}( g \odot e_{N}M) = g G_{M} g^{-1} $, and that $ \ker(\lambda_{\odot}) = \bigcap_{g \in G} gG_{M}g^{-1} $, the largest normal subgroup of $ G $ contained in $ G_{M} $. We note that a quotient of a reduced skew bracoid need not be reduced:

\begin{example} \label{eg_running_quotient_not_reduced}
Consider the skew bracoid $ (G,N) $ constructed in Example \ref{eg_running_first_principles} in the particular case in which $ n=d=4 $. Recall from Example \ref{eg_running_kernel_of_lambda} that $ (G,N) $ is reduced, from Example \ref{eg_running_ideals} that $ M = \langle \eta^{2} \rangle $ is an ideal of $ (G,N) $, and from Example \ref{eg_quotient_ideals_G_M} that $ G_{M} = \langle r^{2}, s \rangle $. Hence $ \Stab_{G}(e_{N}M) = G_{M} $ is a normal subgroup of $ G $, and so the kernel of the action of $ G $ on $ N/M $ is also equal to $ G_{M} $. Therefore $ (G, N/M) $ is not reduced.
\end{example}

We note that in Example \ref{eg_running_quotient_not_reduced} we have $ |G|=8, |N|=4, |M|=2, $ and $ |G_{M}|=4 $. Since $ G_{M} $ coincides with the kernel of the action of $ G $ on $ N/M $, the reduced form $ (\bar{G},N/M) $ of this skew bracoid (see Proposition \ref{prop_reduced_form}) satisfies $ |\bar{G}| = |N/M| = 2 $, and so is essentially a skew brace. To study this phenomenon further we make a definition:

\begin{definition} \label{defn_enhanced_left_ideal}
An \textit{enhanced left ideal} (resp. \textit{enhanced ideal}) of a skew bracoid $ (G,N) $ is a left ideal (resp. ideal) $ M $ such that $ G_{M} $ is normal in $ G $. 
\end{definition}

In the case of skew braces, an enhanced ideal of a skew brace $ (B,\star,\cdot) $ is a left ideal $ A $ that is a normal subgroup of $ (B,\cdot) $; we are not aware of a term for this construction elsewhere in the literature. 

As in Proposition \ref{prop_left_ideals_and_reduced_forms}, a subset $ M $ of $ N $ is an enhanced left ideal, or enhanced ideal, of a skew bracoid $ (G,N) $ if and only if it is such a substructure of $ (\bar{G},N) $. Enhanced ideals play an important role in the connection between skew bracoids and Hopf-Galois structures described in Section \ref{sec_HGS}. 

\begin{proposition}
Let $ (G,N) $ be a skew bracoid and let $ M $ be an ideal of $ (G,N) $. Then $ M $ is an enhanced ideal of $ (G,N) $ if and only if the reduced form of the skew bracoid $ (G,N/M) $ is essentially a skew brace. 
\end{proposition}
\begin{proof}
First suppose that $ M $ is an enhanced ideal of $ (G,N) $, so that $ G_{M} $ is normal in $ G $. As noted above, in the skew bracoid $ (G, N/M) $ we have $ \Stab_{G}(g \odot M) = gG_{M}g^{-1} $ for each $ g \in G $. Since $ G_{M} $ is normal in $ G $ these stabilizers all coincide, so $ \ker(\lambda_{\odot}) = G_{M} $, and so the reduced form of $ (G, N/M) $ is $ ( G/G_{M}, N/M ) $. Now let $ S = \Stab_{G}(e_{N}) $, so that $ |G| = |S||N| $. By the definition of $ G_{M} $ (see Proposition \ref{prop_left_ideal_sub_SBO}) we have $ S \subseteq G_{M} $, and so $ |G_{M}|=|S||M| $.  Therefore we have
\[ \left| \frac{G}{G_{M}} \right| = \frac{ |G| }{ |G_{M}| } = \frac{ |S||N| }{ |S||M| } = \frac{ |N| }{ |M| } = \left| \frac{ N }{M} \right|, \]
and so $ ( G/G_{M}, N/M ) $ is essentially a skew brace. 

Conversely, suppose that the reduced form of $ (G, N/M) $ is essentially a skew brace. Then, writing $ K $ for the kernel of the action of $ G $ on $ N/M $, we have $ |G| / |K| = |N| / |M| $, and so $ |K| = |G_{M}| $. But $ K \subseteq S \subseteq G_{M} $, so in fact $ K = G_{M} $. Therefore $ G_{M} $ is normal in $ G $, and so $ M $ is an enhanced ideal of $ (G,N) $. 
\end{proof}

\begin{proposition}
Let $ (G,N) $ be a skew bracoid, and let $ M $ be an ideal of $ (G,N) $. There is a bijective correspondence between left ideals, ideals, enhanced left ideals, enhanced ideals of $ (G,N/M) $ and the corresponding substructures of $ (G,N) $ that contain $ M $.
\end{proposition}
\begin{proof}
There is a bijective correspondence between subgroups of $ N/M $ and subgroups of $ N $ that contain $ M $. Let $ P $ be a subgroup of $ N $ that contains $ M $, with corresponding subgroup $ P/M $ of $ N/M $. 

Recall that the $ \gamma $-function of the skew bracoid $ (G,N/M) $ is given by $ \,^{\gamma(g)}(\eta M) = (^{\gamma(g)}\eta) M $ for all $ \eta \in N $ and $ g \in G $. Therefore we have $ \,^{\gamma(g)}(\pi M) \in P/M $ for all $ \pi \in P $ if and only if $ (\,^{\gamma(g)}\pi) M \in P/M $ for all $ \pi \in P $ if and only if $ \,^{\gamma(g)}\pi \in P $ for all $ \pi \in P $, and so $ P $ is a left ideal of $ (G,N) $ if and only if $ P/M $ is a left ideal of $ (G, P/M) $. 

The correspondence between subgroups of $ N/M $ and subgroups of $ N $ that contain $ M $ respects and detects normality, so a left ideal of $ (G,N) $ that contains $ M $ is an ideal if and only if the corresponding left ideal of $ (G,N/M) $ is an ideal. 

Finally, since the action of $ G $ on $ N/M $ is given by $ g \odot (\eta M) = (g \odot \eta)M $, we see that if $ P $ is a left ideal of $ (G,N) $ then we have $ G_{P/M} = G_{P} $, and so $ P $ is enhanced if and only if $ P/M $ is enhanced.  
\end{proof}

\section{Homomorphisms and Isomorphisms} \label{sec_homomorphisms}

A homomorphism of skew braces is simply a map between the underlying sets that respects both of the group operations. The kernel of a skew brace homomorphism is an ideal of the domain, the image is a subskew brace of the codomain, and a version of the first isomorphism theorem holds. In this section we generalize these results to skew bracoids. 

\begin{definition} \label{defn_homomorphism}
A \textit{homomorphism} of skew bracoids $ (G,N,\odot) \rightarrow (G', N', \odot') $ is a pair of group homomorphisms $ \varphi : G \rightarrow G' $ and $ \psi : N \rightarrow N' $ such that 
\begin{equation} \label{eqn_homomorphism_relation}
 \psi(g \odot \eta) = \varphi(g) \odot' \psi(\eta)
\end{equation}
for all $ g \in G $ and $ \eta \in N $. 
\end{definition}

The presence of two different group homomorphisms in Definition \ref{defn_homomorphism} makes it appear rather unwieldy. In fact, skew bracoid homomorphisms are more restricted then they first appear:

\begin{proposition} \label{prop_homomorphism}
Let $ (G,N,\odot) $ and $ (G', N', \odot') $ be skew bracoids, let $ S = \Stab_{G}(e_{N}) $, let $ S' = \Stab_{G'}(e_{N'}) $, and let $ \varphi : G \rightarrow G' $ be a group homomorphism.

If $ \varphi(S) \subseteq S' $ and the map $ \varphi_{N} : N \rightarrow N' $ defined by  
\begin{equation} \label{eqn_homomorphism_relation_2} 
\varphi_{N}(g \odot e_{N}) = \varphi(g) \odot' e_{N'} 
\end{equation}
is a group homomorphism then $ \varphi, \varphi_{N} $ form a homomorphism of skew bracoids. 

Conversely, if $ \psi : N \rightarrow N' $ is a group homomorphism such that $ \varphi, \psi $ form a homomorphism of skew bracoids then $ \varphi(S) \subseteq S' $ and $ \psi = \varphi_{N} $. 
\end{proposition}
\begin{proof}
First suppose that $ \varphi(S) \subseteq S' $ and that $ \varphi_{N} : N \rightarrow N' $ is a group homomorphism. (Note that the assumption that $ \varphi(S) \subseteq S' $ and the fact that $ G $ acts transitively on $ N $ ensure that $ \varphi_{N} $ is well defined.) Since $ G $ acts transitively on $ N $, given $ \eta \in N $ we may write $ \eta = h \odot e_{N} $ for some $ h \in G $, and then we have
\begin{eqnarray*}
\varphi_{N}( g \odot \eta) & = & \varphi_{N}( g \odot (h \odot e_{N}) ) \\
& = & \varphi_{N}( gh \odot e_{N} ) \\
&= & \varphi(gh) \odot' e_{N'} \mbox{ by \eqref{eqn_homomorphism_relation_2} }\\
& = & \varphi(g) \odot' (\varphi(h) \odot' e_{N'}) \mbox{ since $ \varphi $ is a homomorphism }\\
& = & \varphi(g) \odot' \varphi_{N}(h \odot e_{N}) \mbox{ by \eqref{eqn_homomorphism_relation_2} }\\
& = & \varphi(g) \odot' \varphi_{N}(\eta).
\end{eqnarray*}
Thus \eqref{eqn_homomorphism_relation} is satisfied, and so $ \varphi, \varphi_{N} $ form a homomorphism of skew bracoids. 
Conversely, suppose that $ \psi : N \rightarrow N' $ is a group homomorphism such that $ \varphi, \psi $ form a homomorphism of skew bracoids. If $ g \in S $ then we have
\begin{eqnarray*}
\varphi(g) \odot' e_{N'} & = & \varphi(g) \odot' \psi(e_{N}) \\
& = & \psi( g \odot e_{N} ) \\
& = & \psi( e_{N} ) \\
& = & e_{N'}.
\end{eqnarray*} 
Thus $ \varphi(S) \subseteq S' $, and so $ \varphi_{N} $ is well defined. Since $ G $ acts transitively on $ N $, given $ \eta \in N $ we may write $ \eta = h \odot e_{N} $ for some $ h \in G $, and then we have 
\begin{eqnarray*}
\psi(\eta) & = & \psi(h \odot e_{N}) \\
& = & \varphi(h) \odot' e_{N'} \\
& = & \varphi_{N}(h \odot e_{N}) \\
& = & \varphi_{N}(\eta). 
\end{eqnarray*}
Thus $ \psi = \varphi_{N} $, which completes the proof. 
\end{proof}

In light of Proposition \ref{prop_homomorphism} we will adopt the following notation for homomorphisms of skew bracoids

\begin{notation}
We will write $ \boldvarphi : (G,N,\odot) \rightarrow (G', N', \odot') $ to denote the homomorphism of skew bracoids comprised of the group homomorphisms $ \varphi : G \rightarrow G' $ and $ \varphi_{N} : N \rightarrow N' $. 
\end{notation}

\begin{example}
If we view skew braces as skew bracoids then a homomorphism of skew braces is a homomorphism of skew bracoids in which the two maps coincide. 

If $ (B, B/A) $ is a skew bracoid arising from a skew brace $ (B,\star,\cdot) $ and a strong left ideal $ A $, as described in Proposition \ref{prop_skew_bracoid_quotient_skew_brace}, then we obtain a natural homomorphism of skew bracoids $ (B,B) \rightarrow (B, B/A) $ as follows: let $ \varphi : B \rightarrow B $ be the identity map. In the notation of Proposition \ref{prop_homomorphism} we have $ S = \{ e_{B} \} $ and $ S' = A $, so $ \varphi(S) \subseteq S' $. The map $ \varphi_{B} : B \rightarrow B/A $ is given by 
\[ \varphi_{B}(b) = b \odot A = b \star A, \]
which is the natural projection $ (B, \star) \twoheadrightarrow (B/A, \star) $. Therefore $ \varphi_{B} $ is a homomorphism, and so $ \varphi, \varphi_{B} $ form a homomorphism of skew bracoids. 

More generally, if $ (G,N) $ is a skew bracoid and $ M $ is an ideal of $ (G,N) $ then choosing $ \varphi : G \rightarrow G $ to be the identity map we find that $ \varphi_{N} : N \twoheadrightarrow N/M $ is the natural projection and $ \varphi, \varphi_{N} $ form a homomorphism of skew bracoids. 
\end{example}

Next we show that kernels and images of skew bracoid homomorphisms have the properties we would naturally expect. 

\begin{proposition} \label{prop_homomorphism_kernel_image}
Let $ \boldvarphi : (G,N,\odot) \rightarrow (G', N', \odot') $ be a homomorphism of skew bracoids. Then 
\begin{enumerate}
\item $ \ker(\varphi_{N}) $ is an ideal of $ (G,N) $;
\item $ (\Im(\varphi), \Im(\varphi_{N})) $ is a subskew bracoid of $ (G', N') $.
\end{enumerate}
\end{proposition}
\begin{proof}
\begin{enumerate}
\item It is clear that $ \ker(\varphi_{N}) $ is a normal subgroup of $ N $. If $ \eta \in \ker(\varphi_{N}) $ and $ g \in G $ then we have
\begin{eqnarray*}
\varphi_{N}( \,^{\gamma(g)} \eta ) & = & \varphi_{N}( (g \odot e_{N})^{-1} \star (g \odot \eta) ) \\
& = & \varphi_{N}( g \odot e_{N})^{-1} \star' \varphi_{N}(g \odot \eta) \\
& = &  (\varphi(g) \odot' \varphi_{N}(e_{N}) )^{-1} \star' (\varphi(g) \odot \varphi_{N}(\eta)) \\
& = &  (\varphi(g) \odot' e_{N'})^{-1} \star' (\varphi(g) \odot' e_{N'}) \\
& = & e_{N'}.
\end{eqnarray*}
Therefore $ \ker(\varphi_{N}) $ is an ideal of $ (G,N) $. 
\item It is clear that $ \Im(\varphi) $ is a subgroup of $ G' $ and that $ \Im(\varphi_{N}) $ is a subgroup of $ N' $. If $ g \in G $ and $ \eta \in N $ then we have
\[ \varphi(g) \odot' \varphi_{N}(\eta) = \varphi_{N}( g \odot \eta ) \in \Im(\varphi_{N}), \]
so $ \Im(\varphi) $ acts on $ \Im(\varphi_{N}) $. To show that this action is transitive, let $ \varphi_{N}(\eta) \in \Im(\varphi_{N}) $, with $ \eta \in N $. Since $ G $ acts transitively on $ N $ there exists $ g \in G $ such that $ g \odot e_{N} = \eta $, and we have
\begin{eqnarray*}
\varphi_{N}(\eta) & = & \varphi_{N}(g \odot e_{N}) \\
& = & \varphi(g) \odot' \varphi_{N}(e_{N}) \\
& = & \varphi(g) \odot' e_{N'}. 
\end{eqnarray*} 
Therefore $ \Im(\varphi) $ acts transitively on $ \Im(\varphi_{N}) $. The skew bracoid relation is satisfied because it holds in $ (G',N') $. Therefore $ (\Im(\varphi), \Im(\varphi_{N})) $ is a subskew bracoid of $ (G', N') $.
\end{enumerate}
\end{proof}

\begin{example}
Let $ (G,N) $ be a skew bracoid, let $ M $ be an ideal of $ (G,N) $, and consider the skew bracoid $ (G,N/M) $. Let $ \varphi : G \rightarrow G $ be the identity map, so that $ \varphi_{N} : N \twoheadrightarrow N/M $ is the natural projection. Then $ \ker(\varphi_{N}) = M $ and $ (\Im(\varphi), \Im(\varphi_{N})) = (G,N/M) $. 
\end{example}

\begin{definition}
A homomorphism of skew bracoids $ \boldvarphi : (G,N,\odot) \rightarrow (G',N',\odot') $ is called a \textit{isomorphism} if $ \varphi $ and $ \varphi_{N} $ are isomorphisms. 
\end{definition}

\begin{example} \label{eg_running_isomorphic_quotient}
Consider the skew bracoid $ (G,N) $ constructed in Example \ref{eg_running_first_principles}. We can formalize our observation from Example \ref{eg_running_quotient_of_skew_brace} that this skew bracoid can be obtained as the quotient of a skew brace by a strong left ideal. 

Let $ (G,\star,\cdot) $ be the skew brace described in Example \ref{eg_running_quotient_of_skew_brace}, so that $ (G,\cdot) \cong D_{n} $ and $ (G,\star) \cong C_{n} \times C_{2} $, and consider the strong left ideal $ H = \langle r^{d}, s \rangle $ and the skew bracoid $ (G, G/H) $. We construct an isomorphism of skew bracoids $ \boldvarphi : (G,N) \rightarrow (G,G/H) $ as follows.

Let $ \varphi : G \rightarrow G $ be the identity map. Note that $ \Stab_{G}(e_{N}) = \Stab_{G}(e_{G}H) = \langle r^{d}, s \rangle $ so, in the notation of Proposition \ref{prop_homomorphism} we have  $ \varphi(S) \subseteq S' $. Therefore $ \varphi $ induces a map $ \varphi_{N} : N \rightarrow G/H $, which is given by 
\[ \varphi_{N}(\eta^{k}) = \varphi_{N}(r^{k} \odot e_{N}) = \varphi(r^{k}) \odot e_{G}H = r^{k} \odot e_{G}H = r^{k}H = (rH)^{k}. \]
Therefore $ \varphi_{N} $ is an isomorphism from $ N $ to $ G/H $, and so $ \boldvarphi : (G,N) \rightarrow (G,G/H) $ is an isomorphism of skew bracoids. 
\end{example}

The following Proposition strengthens Proposition \ref{prop_homomorphism} in the case of isomorphisms.

\begin{proposition} \label{prop_isomorphism_stabilizers}
Let $ \boldvarphi : (G,N,\odot) \rightarrow (G',N',\odot') $ be a homomorphism of skew bracoids.

If $ \boldvarphi $ is an isomorphism of skew bracoids then for each $ \eta \in N $ we have $ \Stab_{G'}(\varphi_{N}(\eta)) = \varphi(\Stab_{G}(\eta)) $. 

Conversely, if $ \varphi : G \rightarrow G $ is an isomorphism and $ \Stab_{G'}(e_{N'}) = \varphi(\Stab_{G}(e_{N})) $ then $ \boldvarphi $ is an isomorphism of skew bracoids.
\end{proposition}
\begin{proof}
First suppose that $ \boldvarphi $ is an isomorphism. Since $ \varphi $ and $ \varphi_{N} $ are isomorphisms, for $ g \in G $ and $ \eta \in N $ we have
\begin{eqnarray*}
&&  g \odot \eta = \eta \\
& \Leftrightarrow & \varphi_{N}( g \odot \eta ) = \varphi_{N}(\eta) \\
& \Leftrightarrow & \varphi(g) \odot' \varphi_{N}(\eta) = \varphi_{N}(\eta).
\end{eqnarray*}
Thus $ \Stab_{G'}(\varphi_{N}(\eta)) = \varphi(\Stab_{G}(\eta)) $, as claimed. 

Now suppose that $ \varphi : G \rightarrow G' $ is an isomorphism and $ \Stab_{G'}(e_{N'}) = \varphi(\Stab_{G}(e_{N})) $. We shall show that the homomorphism $ \varphi_{N} : N \rightarrow N' $ is also an isomorphism. We note that
\[ |N| = \frac{ |G| }{ |\Stab_{G}(e_{N})| } = \frac{ |\varphi(G)| }{ |\varphi(\Stab_{G}(e_{N}))| } = \frac{ |G'| }{ |\Stab_{G'}(e_{N'})| } = |N'|. \]
Now let $ \eta \in N $, and write $ \eta = g \odot e_{N} $ with $ g \in G $. Then 
\[ \varphi_{N}(\eta) = \varphi_{N}(g \odot e_{N}) = \varphi(g) \odot' \varphi_{N}(e_{N}) = \varphi(g) \odot' e_{N'}, \]
so $ \eta \in \ker(\varphi_{N}) $ if and only if $ \varphi(g) \in \Stab_{G'}(e_{N'}) $. Since $ \Stab_{G'}(e_{N'}) = \varphi(\Stab_{G}(e_{N})) $, this occurs if and only if $ \eta = e_{N} $. Therefore $ \varphi_{N} $ is injective, hence bijective. This completes the proof. 
\end{proof}

Since $ \ker(\lambda_{\odot}) = \bigcap_{\eta \in N} \Stab_{G}(\eta) $, and similarly for $  \ker(\lambda_{\odot'}) $, we obtain

\begin{corollary} \label{cor_isomorphic_kernel_of_action}
If $ \boldvarphi : (G,N,\odot) \rightarrow (G',N',\odot') $ is an isomorphism of skew bracoids then $  \ker(\lambda_{\odot'}) = \varphi(\ker(\lambda_{\odot})) $. 
\end{corollary}

By combining many of the results of this section, we prove a version of the First Isomorphism Theorem for skew bracoids. 

\begin{theorem} \label{thm_FIT}
Suppose that $ \boldvarphi : (G,N,\odot) \rightarrow (G',N',\odot') $ is a homomorphism of skew bracoids, and let $ M = \ker(\varphi_{N}) $. Then the reduced forms of $ (G, N/M) $ and $ (\Im(\varphi), \Im(\varphi_{N})) $ are isomorphic skew bracoids. 
\end{theorem}
\begin{proof}
By Proposition \ref{prop_homomorphism_kernel_image} $ (\Im(\varphi), \Im(\varphi_{N})) $ is a subskew bracoid of $ (G',N') $; to ease notation we will relabel (without loss of generality) so that $ \Im(\varphi)=G' $ and $ \Im(\varphi_{N})) = N' $. 

With this relabelling, the group homomorphism $ \varphi_{N} : N \rightarrow N' $ is a surjection, and induces an isomorphism $ N/M \rightarrow N' $; we denote this by $ \varphi_{N/M} $. 

By Proposition \ref{prop_homomorphism_kernel_image} we see that $ M $ is an ideal of $ (G,N) $, and so $ (G,N/M) $ is a skew bracoid (Proposition \ref{prop_quotient_skew_bracoid}). For all $ g \in G $ and $ \eta M \in N/M $ we have
\[ \varphi(g) \odot' \varphi_{N/M}(\eta M) = \varphi(g) \odot' \varphi_{N}(\eta) = \varphi_{N}(g \odot \eta) = \varphi_{N/M}(g \odot \eta M), \]
so (abusing notation) we obtain a homomorphism of skew bracoids $ \boldvarphi : (G,N/M,\odot) \rightarrow (G',N',\odot') $ in which $ \varphi : G \rightarrow G' $ is surjective and $ \varphi_{N/M} : N/M \rightarrow N' $ is an isomorphism. 

Now let $ K $ denote the kernel of the action of $ G $ on $ N/M $ and let $ K' $ denote the kernel of the action of $ G' $ on $ N' $. 

Let $ \theta : G \rightarrow G'/K' $ be the composition of $ \varphi : G \rightarrow G' $ with the natural projection; recalling our relabelling $ \Im(\varphi)=G' $, we see that $ \theta $ is a surjection. We have
\begin{eqnarray*}
g \in \ker(\theta) & \Leftrightarrow & \varphi(g) \in K' \\
& \Leftrightarrow & \varphi(g) \odot' \varphi_{N/M}(\eta M) = \varphi_{N/M}(\eta M) \mbox{ for all } \eta M \in N/M \\
& \Leftrightarrow & \varphi_{N/M}( g \odot \eta M) = \varphi_{N/M}(\eta M) \mbox{ for all } \eta M \in N/M \\
& \Leftrightarrow & g \odot \eta M = \eta M \mbox{ for all } \eta M \in N/M \mbox{ since $ \varphi_{N/M} $ is an isomorphism }\\
& \Leftrightarrow & g \in K. 
\end{eqnarray*}
Hence the surjection $ \theta : G \rightarrow G'/K' $ induces an isomorphism $ \theta :G/K \rightarrow G'/K' $. Since $ \varphi : G \rightarrow G' $ maps $ \Stab_{G}(e_{N}M) $ into $ \Stab_{G'}(e_{N'}) $, the map $ \theta $ also has this property; it therefore induces a map $ \theta_{N/M} : N/M \rightarrow N' $. This is given by 
\[ \theta_{N/M}( g \odot e_{N} M ) = \theta(g) \odot' e_{N'} = \varphi(g) \odot' e_{N'} = \varphi_{N/M}(g \odot e_{N} M). \]
Since $ \varphi_{N/M} : N/M \rightarrow N' $ is an isomorphism, we see that $ \theta_{N/M} $ is an isomorphism, and so $ \theta, \theta_{N/M} $ form an isomorphism of skew bracoids $ (\bar{G}, N/M) \rightarrow (\bar{G'}, N') $. 
\end{proof}

\begin{corollary} \label{cor_if_isomorphic_then_reduced_forms_isomorphic}
If $ (G,N,\odot) \cong (G',N',\odot') $ then $ (\bar{G},N,\odot) \cong (\bar{G'},N',\odot') $. 
\end{corollary}
\begin{proof}
If $ \boldvarphi : (G,N,\odot) \rightarrow (G',N,\odot') $ is an isomorphism then $ \varphi_{N} : N \rightarrow N' $ is an isomorphism, so in the notation of Theorem \ref{thm_FIT} we have $ M = \{ e_{N} \} $. Now Theorem \ref{thm_FIT} implies that $ (\bar{G},N,\odot) \cong (\bar{G'},N',\odot') $. 
\end{proof}

We have already recalled the result that if $ (B,\star) $ is a group and $ \cdot, \cdot' $ are two further binary operations on $ B $ such that $ (B,\star,\cdot) $ and $ (B,\star,\cdot') $ are skew braces, then we have $ (B,\star,\cdot) = (B,\star,\cdot') $ if and only if $ \lambda_{\bullet}(B) = \lambda_{\bullet'}(B) $ inside $ \Hol_{\star}(B) $; this motivated our definition of equivalence of skew bracoids (Definition \ref{defn_equivalent}). 

We now recall that, furthermore, we have $ (B,\star,\cdot) \cong (B,\star,\cdot') $ if and only if there exists $ \theta \in \Aut_{\star}(B) $ such that $ \lambda_{\bullet'}(B) = \theta \lambda_{\bullet}(B) \theta^{-1} $. We can detect isomorphisms between reduced skew bracoids in a similar way. 

\begin{proposition}
Let $ N $ be a group and let $ (G,N,\odot) $ and  $ (G',N,\odot') $ be reduced skew bracoids. Then $ (G,N,\odot) \cong (G',N,\odot') $ if and only if there exists $ \theta \in \Aut(N) $ such that
\[ \lambda_{\odot'}(G') = \theta \lambda_{\odot}(G) \theta^{-1} \subseteq \Hol(N). \]
\end{proposition}
\begin{proof}
First suppose that $ \boldvarphi: (G,N,\odot) \rightarrow (G',N,\odot') $ is an isomorphism of skew bracoids. Then $ \varphi_{N} : N \rightarrow  N $ is an automorphism of $ N $, and for all $ g \in G $ and $ \eta \in N $ we have
\[ \varphi(g) \odot' \varphi_{N}(\eta) = \varphi_{N}(g \odot \eta). \]
Thus for all $ g \in G $ we have
\[ \varphi_{N}^{-1} \lambda_{\odot'}(g) \varphi_{N} = \lambda_{\odot}(g), \]
and so, choosing $ \theta = \varphi_{N} \in \Aut(N) $, we have $ \lambda_{\odot'}(G') = \theta \lambda_{\odot}(G) \theta^{-1} $. 

Conversely suppose that there exists $ \theta \in \Aut(N) $ such that $ \lambda_{\odot'}(G') = \theta \lambda_{\odot}(G) \theta^{-1} $. Then for each $ g \in G $ there exists a unique element $ \varphi(g) \in G' $ such that 
\[ \theta \lambda_{\odot}(g) \theta^{-1} = \lambda_{\odot'}(\varphi(g)). \]
We see that $ \varphi : G \rightarrow G' $ is an isomorphism and that $ \Stab_{G'}(e_{N}) = \varphi(\Stab_{G}(e_{N})) $, so $ \varphi $ induces a bijection $ \varphi_{N} : N \rightarrow N $ defined by $ \varphi_{N}(g \odot e_{N}) = \varphi_{N}(g) \odot' e_{N'} $. Now for each $ g \in G $ we have
\begin{eqnarray*}
\varphi_{N}(g \odot e_{N}) & = & \varphi(g) \odot' e_{N'} \\
& = & \theta \lambda_{\odot}(g) \theta^{-1} [e_{N}] \\
& = & \theta(g \odot e_{N}).
\end{eqnarray*}
Thus $ \varphi_{N} = \theta \in \Aut(N) $, and so $ \varphi, \varphi_{N} $ form an isomorphism of skew bracoids. Hence $ (G,N,\odot) \cong (G',N,\odot') $.
\end{proof}

\begin{corollary}
Let $ N $ be a group and let $ (G,N,\odot) $ be a reduced skew bracoid. Then the number of equivalence classes of skew bracoids that are isomorphic to $ (G,N,\odot) $ is equal to 
\[ \frac{ |\Aut(N)| }{ |\Aut_{\odot}(N)| }, \]
where $ \Aut_{\odot}(N) = \{ \theta \in \Aut(N) \mid \theta \mbox{ normalizes } \lambda_{\odot}(G) \} $. 
\end{corollary}

\begin{corollary}
Let $ N $ be a group and let $ (G,N,\odot) $ and  $ (G',N,\odot') $ be skew bracoids. Then the reduced forms of $ (G,N,\odot) $ and $ (G',N,\odot') $ are isomorphic if and only if there exists $ \theta \in \Aut(N) $ such that
\[ \lambda_{\odot'}(G') = \theta \lambda_{\odot}(G) \theta^{-1} \subseteq \Hol(N). \]
\end{corollary}

We conclude this section by exploring some more interactions between isomorphism and equivalence of skew bracoids. 

\begin{proposition}
Let $ N $ be a group. Then two skew bracoids $ (G,N,\odot) $ and $ (G',N,\odot') $ are equivalent if and only if there is an isomorphism $ \boldvarphi : (\bar{G},N,\odot) \rightarrow (\bar{G'},N,\odot') $ such that $ \varphi_{N} : N \rightarrow N $ is the identity map.
\end{proposition}
\begin{proof}
First suppose that $ (G,N,\odot) $ and $ (G',N,\odot') $ are equivalent, so that $ \lambda_{\odot}(G) = \lambda_{\odot'}(G') \subseteq \Hol(N) $. Passing to the reduced forms we also have $ \lambda_{\odot}(\bar{G}) = \lambda_{\odot'}(\bar{G'}) \subseteq \Hol(N) $, with each of $ \lambda_{\odot} $ and $ \lambda_{\odot'} $ now being injective. The map $ \varphi : \bar{G} \rightarrow \bar{G'} $ defined by $ \varphi = \lambda_{\odot'}^{-1} \lambda_{\odot} $ is therefore an isomorphism, and for all $ \bar{g} \in \bar{G} $ we have 
\[ \varphi(\bar{g}) \odot' e_{N} = \bar{g} \odot e_{N}. \]
Therefore $ \varphi(\bar{g})  $ stabilizes $ e_{N} $ if and only if $ \bar{g} $ does, and so $ \varphi $ induces a map $ \varphi_{N} : N \rightarrow N $, which is given by
\[ \varphi_{N}( \bar{g} \odot e_{N} ) = \varphi(\bar{g}) \odot' e_{N} = \bar{g} \odot e_{N} \mbox{ for all } \bar{g} \in \bar{G}. \]
Therefore $ \varphi_{N} : N \rightarrow N $ is the identity map, and so $ \boldvarphi $ is an isomorphism of skew bracoids of the form given in the proposition.

Conversely, suppose that there is an isomorphism $ \boldvarphi : (\bar{G},N,\odot) \rightarrow (\bar{G'},N,\odot') $ such that $ \varphi_{N} : N \rightarrow N $ is the identity map. Then for all $ \bar{g} \in \bar{G} $ and all $ \eta \in N $ we have
\[ \bar{g} \odot \eta = \varphi_{N}(\bar{g} \odot \eta) = \varphi(\bar{g}) \odot' \varphi_{N}(\eta) = \varphi(\bar{g}) \odot' \eta, \]
and so $ \lambda_{\odot}(\bar{G}) = \lambda_{\odot'}(\bar{G'}) \subseteq \Hol(N) $. But $ \lambda_{\odot}(\bar{G}) = \lambda_{\odot}(G) $ by the definition of $ \bar{G} $, and similarly $ \lambda_{\odot'}(\bar{G'}) = \lambda_{\odot'}(G') $. Therefore $ \lambda_{\odot}(G)=\lambda_{\odot'}(G') $, and so $ (G,N) $ and $ (G',N) $ are equivalent. 
\end{proof}

\section{Connecting skew bracoids with Hopf-Galois structures on separable extensions} \label{sec_HGS}

In Section \ref{sec_introduction} we briefly summarized the connection between skew braces and Hopf-Galois structures on Galois field extensions. In this section we generalize this to a connection between skew bracoids and Hopf-Galois structures on (finite) separable extensions. We begin by describing in more detail the results we generalize.

A theorem of Greither and Pareigis \cite{GP87} classifies the Hopf-Galois structures admitted by a separable extension of fields, as follows: let $ \widetilde{L} $ be the Galois closure of $ L/K $, let $ J = (J,\cdot) = \Gal(\widetilde{L}/K) $, let $ J' = \Gal(\widetilde{L}/L) $, and consider the left coset space $ J/J' $. Consider the left translation map $ \lambda_{\odot} : J \rightarrow \Perm(J/J') $ defined by $ \lambda_{\odot}(j)[xJ'] = jxJ' $, and the action of $ J $ on $ \Perm(J/J') $ by $ \,^{j} \eta = \lambda_{\odot}(j) \eta \lambda_{\odot}(j)^{-1} $ for all $ j \in J $ and $ \eta \in \Perm(J/J') $. 
There is a bijection between regular subgroups of $ \Perm(J/J') $ stable under this action of $ J $ (\textit{$ J $-stable regular subgroups}) and Hopf-Galois structures on $ L/K $. 

If $ L/K $ is a Galois extension then $ J' $ is trivial, so the Greither Pareigis theorem implies that there is a bijection between regular $ J $-stable subgroups of $ \Perm(J) $ and Hopf-Galois structures on $ L/K $. There are various approaches to showing that $ J $-stable regular subgroups of $ \Perm(J) $ are connected with skew braces; we follow Stefanello and Trappeniers \cite{ST22}. There is a bijection between binary operations $ \star $ on $ J $ such that $ (J,\star) $ is a group and regular subgroups of $ \Perm(J) $, given by 
\begin{equation}
\star \leftrightarrow \rho_{\star}(J), \mbox{ where } \rho_{\star}(j)[x] = x \star j^{-1}. 
\end{equation}
Recalling that $ \cdot $ denotes the original binary operation on $ J $, we find that $ (J,\star,\cdot) $ is a skew brace if and only if $ \rho_{\star}(J) $ is $ J $-stable, which occurs if and only if it yields a Hopf-Galois structure on $ L/K $. In this way Stefanello and Trappeniers obtain a bijection between binary operations $ \star $ on $ J $ such that $ (J,\star,\cdot) $ is a skew brace and Hopf-Galois structures on $ L/K $. 

Returning to the case in which $ L/K $ is separable, but possibly non-normal, we observe that we may replace the Galois closure $ \widetilde{L} $ in the statement of the Greither-Pareigis theorem with any (finite) Galois extension $ E $ of $ K $ that contains $ L $. To see this, let $ G = \Gal(E/K) $, $ \widetilde{G} = \Gal(E/\widetilde{L}) $, and $ G' = \Gal(E/L) $ (note that $ \widetilde{G} $ is a normal subgroup of $ G $). Then $ J \cong G/\widetilde{G} $, $ J' \cong G'/\widetilde{G} $, and there is a natural identification of coset spaces $ G / G'  \rightarrow J / J' $ given by $ gG' \mapsto (g\widetilde{G})J' $. Moreover, the left translation map $ \lambda_{\odot} : G \rightarrow \Perm(G/G') $ has kernel $ \widetilde{G} $, and so factors through $ J $. Therefore there is a bijection between $ G $-stable regular subgroups of $ \Perm(G / G') $ and $ J $-stable regular subgroups of $ \Perm(J/J') $. 

Therefore, throughout this section we denote by $ E/K $ a Galois extension of fields with Galois group $ (G,\cdot) $, by $ L $ an intermediate field of $ E/K $ with corresponding subgroup $ G' \subseteq G $, and by $ X $ the left coset space $ G/G' $, in which we write $ \bar{x} $ for the coset  $ xG' $.

We then obtain the following generalization of the result of Stefanello and Trappeniers \cite{ST22}:

\begin{theorem} \label{thm_Hopf_Galois_structures}
There are bijections between
\begin{enumerate}
\item binary operations $ \star $ on $ X $ such that $ (G,\cdot,X,\star,\odot) $ is a skew bracoid, where $ \odot $ denotes left translation of cosets;
\item$ G $-stable regular subgroups of $ \Perm(X) $;
\item Hopf-Galois structures on $ L/K $. 
\end{enumerate}
\end{theorem}
\begin{proof}
The fact that there is a bijection between the objects given in (ii) and (iii) is the content of the theorem of Greither and Pareigis, combined with the discussion above.

Now suppose that $ \star $ is a binary operation on $ X $ such that $ (G,\cdot,X,\star,\odot) $ is a skew bracoid. Let $ \rho_{\star} : X \rightarrow \Perm(X) $ denote the right regular representation of the group $ (X,\star) $, so that $ \rho_{\star}(\bar{x})[\bar{y}] = \bar{y} \star \bar{x}^{-1} $ for all $ \bar{x}, \bar{y} \in X $. Then $ \rho_{\star}(X) $ is clearly a regular subgroup of $ \Perm(X) $. To show that it is $ G $-stable, let $ g \in G $ and $ \bar{x}, \bar{y} \in X $. Then
\begin{eqnarray*}
\lambda_{\odot}(g) \rho_{\star}(\bar{x}) \lambda_{\odot}(g^{-1})[\bar{y}] & = & g \odot ( (g^{-1} \odot \bar{y}) \star \bar{x}^{-1} ) \\
& = & (g \odot (g^{-1} \odot \bar{y}) \star (g \odot \bar{e})^{-1} \star (g \odot \bar{x}^{-1} ) \mbox{ by Equation \ref{eqn_SBO_relation} } \\
& = &  \bar{y} \star (g \odot \bar{x} )^{-1} \star (g \odot \bar{e}) \mbox{ by Proposition \ref{prop_useful_identites} } \\
& = &  \rho_{\star}( (g \odot \bar{e})^{-1} \star (g \odot \bar{x}) ) [\bar{y}] \\
& = & \rho_{\star}( \,^{\gamma(g)} \bar{x} ) [\bar{y}].
\end{eqnarray*} 
Thus $ \rho_{\star}(X) $ is $ G $-stable; in fact, we have shown that 
\begin{equation} \label{eqn_GP_action_on_rho_X}
\lambda_{\odot}(g) \rho_{\star}(\bar{x}) \lambda_{\odot}(g^{-1}) = \rho_{\star}( \,^{\gamma(g)} \bar{x} )
\end{equation}
for all $ g \in G $ and $ \bar{x} \in X $. Thus we obtain a $ G $-stable regular subgroup of $ \Perm(X) $. 

Next suppose that $ N $ is a $ G $-stable regular subgroup of $ \Perm(X) $. Then the map $ a: N \rightarrow X $ defined by $ a(\eta) = \eta^{-1}[\bar{e}] $ is a bijection. Using this, we define a binary operation $ \star $ on $ X $ by the rule
\[ a(\eta) \star a(\mu) = a(\eta \mu). \] 
Then $ (X,\star) $ is a group isomorphic to $ N $. It is clear that the action of $ (G,\cdot) $ on $ X $ by left translation of cosets is transitive, so in order to show that $ (G,\cdot,X,\star,\odot) $ forms a skew bracoid it only remains to show that the skew bracoid relation \eqref{eqn_SBO_relation} is satisfied. To do this, note first that the assumption that $ N $ is $ G $-stable implies that for each $ g \in G $ the map $ \theta_{g} : N \rightarrow N $ defined by $ \theta_{g}(\eta) = \lambda_{\odot}(g) \eta \lambda_{\odot}(g^{-1}) $ is an automorphism of $ N $. For $ \eta \in N $ we have
\begin{eqnarray*}
g \odot a(\eta)  & = & g \odot \eta^{-1}[\bar{e}] \\
& = & \theta_{g}(\eta^{-1})\kappa^{-1} [ \bar{e} ], \mbox{ where $ \kappa \in N $ satisfies $ \kappa^{-1}[\bar{e}] = \bar{g} $ } \\
& = & a(\kappa \theta_{g}(\eta)).
\end{eqnarray*}
Therefore for $ \eta, \mu \in N $ we have
\begin{eqnarray*}
g \odot ( a(\eta) \star a(\mu)) & = & g \odot a(\eta\mu) \\
& = & a(\kappa \theta_{g}(\eta\mu)) \\
& = & a(\kappa \theta_{g}(\eta) \theta_{g}(\mu)) \\
& = & a(\kappa \theta_{g}(\eta) \kappa^{-1} \kappa \theta_{g}(\mu)) \\
& = & a(\kappa \theta_{g}(\eta)) \star a(\kappa)^{-1} \star a(\kappa \theta_{g}(\mu)) \\
& = & (g \odot a(\eta)) \star (g \odot e_{N})^{-1} \star (g \odot a(\mu)). 
\end{eqnarray*}
Therefore the skew bracoid relation \eqref{eqn_SBO_relation} holds, and so $ (G,\cdot,X,\star,\odot) $ forms a skew bracoid. 

The processes described above are mutually inverse: if $ N $ is a $ G $-stable regular subgroup of $ \Perm(X) $ and $ \star $ is the corresponding binary operation on $ X $ then for all $ \eta, \mu \in N $ we have
\begin{eqnarray*}
\rho_{\star}(a(\mu))[a(\eta)] & = & a(\eta) \star a(\mu)^{-1} \\
& = & a(\eta\mu^{-1}) \\
& = & \mu \eta^{-1} [ \bar{e} ] \\
& = & \mu [ a(\eta) ],
\end{eqnarray*}
so $ \rho_{\star}(X) = N \subseteq \Perm(X) $. Similarly, if $ \star $ is a binary operation on $ X $ such that $ (G,\cdot,X,\star,\odot) $ is a skew bracoid and $ N = \rho_{\star}(X) $ is the corresponding $ G $-stable regular subgroup of $ \Perm(X) $ then the bijection $ a : N \rightarrow X $ is given by
\[ a( \rho_{\star}(\bar{x})) = \rho_{\star}^{-1}(\bar{x})[\bar{e}] = \bar{x}, \]
and the binary operation $ \hat{\star} $ on $ X $ arising from $ N $ is given by
\[ \bar{x} \; \hat{\star} \; \bar{y} =  a(\rho_{\star}(\bar{x})) \; \hat{\star} \; a(\rho_{\star}(\bar{y})) =  a(\rho_{\star}(\bar{x} \star \bar{y})) = \bar{x} \star \bar{y}. \]
Hence $ \hat{\star} = \star $. 

This completes the proof that there is a bijection between the objects described in (ii) and those in (iii).
\end{proof}

We note that the bijection $ a $ defined in the proof of Theorem \ref{thm_Hopf_Galois_structures} is not a direct generalization of the bijection used in the corresponding step of the argument connecting skew braces with Hopf-Galois structures on Galois field extensions (see \cite[Subsection 2.2]{KT20}, for example). The direct generalization would be $ a(\eta) = \eta[\bar{e}] $; we choose $ a(\eta) = \eta^{-1}[\bar{e}] $ over this since it yields an isomorphism between $ (X,\star) $ and $ \rho_{\star}(X) $. 

If, instead of beginning with a Galois extension $ E $ of $ K $ and an intermediate field $ L $ we begin with a separable extension $ L/K $, then by Theorem \ref{thm_Hopf_Galois_structures} each Hopf-Galois structure on $ L/K $ yields multiple skew bracoids, corresponding to different choices of the (finite) Galois extension $ E $ of $ K $ containing $ L $. However, we have

\begin{proposition}
Let $ \star $ be a binary operation on $ X $ such that $ (G,\cdot,X,\star,\odot) $ is a skew bracoid. Then $ \ker(\lambda_{\odot}) = \Gal(E/\widetilde{L}) $, so the reduced form of $ (G,X) $ is $ (\bar{G}, X) $ with $ \bar{G} = \Gal(\widetilde{L}/K) $. 
\end{proposition}
\begin{proof}
Recall that the action $ \odot $ on $ G $ on $ X $ is by left translation of cosets. We have $ \Stab_{G}(\bar{e}) = G' $, so for each $ x \in G $ we have $ \Stab_{G}(\bar{x}) = xG'x^{-1} $, and so $ \ker(\lambda_{\odot}) = \bigcap_{x \in G} xG'x^{-1} $, the largest normal subgroup of $ G $ contained in $ G' $. By Galois theory, the fixed field of this subgroup is the smallest subfield of $ E $ that contains $ L $ and is a Galois extension of $ K $; that is, the Galois closure of $ L/K $. Thus $ \ker(\lambda_{\odot}) = \Gal(E/\widetilde{L}) $. The remaining statement follows by the definition of the reduced form of $ (G,X) $ (Proposition \ref{prop_reduced_form}). 
\end{proof}

We obtain immediately

\begin{corollary}
The skew bracoid $ (G,X) $ is reduced if and only if $ E $ is the Galois closure of $ L/K $. 
\end{corollary}

In Greither and Pareigis's original paper \cite{GP87} they show that if $ N $ is a $ G $-stable regular subgroup of $ \Perm(X) $ then so too is $ N^{op} = \mathrm{Cent}_{\tiny \Perm(G)}(N) $. The interactions between the Hopf-Galois structures corresponding to $ N $ and $ N^{op} $ have been explored in, for example, \cite{Tr18b} and \cite{KT20}. In particular, \cite[Proposition 3.4]{KT20} shows that the skew brace corresponding to the subgroup $ N^{op} $ is the opposite of the skew brace corresponding to $ N $. This result generalizes to skew bracoids:

\begin{proposition}
Let $ \star $ be a binary operation on $ X $ such that $ (G,\cdot,X,\star,\odot) $ is a skew bracoid, and let $ (G,\cdot,X,\star^{op},\odot) $ be the opposite skew bracoid. Then $ \rho_{\star^{op}}(X) = \rho_{\star}(X)^{op} $.
\end{proposition}
\begin{proof}
For $ \bar{x}, \bar{y}, \bar{z} \in X $ we have
\begin{eqnarray*}
\rho_{\star^{op}}(\bar{x}) \rho_{\star}(\bar{y}) [\bar{z}] & = & \rho_{\star^{op}}(\bar{x}) [\bar{z} \star \bar{y}^{-1}] \\
& = & \bar{x}^{-1} \star \bar{z} \star \bar{y}^{-1} \\
& = & \rho_{\star}(\bar{y}) [\bar{x}^{-1} \star \bar{z}] \\
& = & \rho_{\star}(\bar{y}) \rho_{\star^{op}}(\bar{x}) [\bar{z}].
\end{eqnarray*}
Hence $ \rho_{\star^{op}}(X) \subseteq \rho_{\star}(X)^{op} $. But these groups have equal order, so in fact they are equal. 
\end{proof}

The connection between skew braces and Hopf-Galois structures on Galois extensions $ L/K $ employed in (for example) \cite{NZ19}, \cite{KT20}, \cite{KST20}, and \cite{KT22} is expressed in terms of \textit{isomorphism classes} of skew braces; in this case multiple Hopf-Galois structures on $ L/K $ can yield a single isomorphism class of skew braces. This is precisely quantified in \cite[Proposition 2.1]{NZ19} and \cite[Proposition 3.1]{KT22}: two $ J $-stable regular subgroups $ N, N' $ of $ \Perm(J) $ yield isomorphic skew braces if and only if $ N' = \varphi N \varphi^{-1} $ for some $ \varphi \in \Aut(J) $. This result generalizes to skew bracoids:

\begin{proposition}\label{prop_isomorphic_braciods_conjugate_in_permx}
Let $ N, N' $ be $ G $-stable regular subgroups of $ \Perm(X) $. The corresponding skew bracoids $ (G,\cdot,X,\star,\odot) $ and $ (G,\cdot,X,\star',\odot) $ are isomorphic if and only if there exists $ \varphi \in \Aut(G) $ such that $ \varphi(G')=G' $ and $ N' = \varphi_{X} N \varphi_{X}^{-1} $. 
\end{proposition}
\begin{proof}
First suppose that $ \boldvarphi : (G,\cdot,X,\star,\odot) \rightarrow (G,\cdot,X,\star',\odot) $ is an isomorphism. Then $ \varphi \in \Aut(G) $, and $ \varphi_{X} : (X,\star) \rightarrow (X,\star') $ is an isomorphism. By Proposition \ref{prop_isomorphism_stabilizers} we have $ \varphi(\Stab_{G}(\bar{e})) = \Stab_{G}(\bar{e}) $, i.e. $ \varphi(G')=G' $. For $ \bar{x}, \bar{y} \in X $ we have
\begin{eqnarray*}
\rho_{\star'}(\bar{x})[\bar{y}] & = & \bar{y} \star' \bar{x}^{-1} \\
& = & \varphi_{X} \left( \varphi_{X}^{-1}(\bar{y}) \star \varphi_{X}^{-1}(\bar{x})^{-1} \right) \\
& = & \varphi_{X} \rho_{\star}(\varphi_{X}^{-1}(\bar{x})) [ \varphi_{X}^{-1}(\bar{y}) ] \\
& = & \varphi_{X} \rho_{\star}(\bar{z}) [ \varphi_{X}^{-1}(\bar{y}) ] 
\end{eqnarray*}
with $ \bar{z} = \varphi_{X}^{-1}(\bar{x}) \in X $. Thus $ \rho_{\star'}(X) = \varphi_{X} \rho_{\star}(X) \varphi_{X}^{-1} $.

Conversely suppose that there exists $ \varphi \in \Aut(G) $ such that $ \varphi(G')=G' $ and $ N' = \varphi_{X} N \varphi_{X}^{-1} $. Recall from Theorem \ref{thm_Hopf_Galois_structures} that the binary operation $ \star $ on $ X $ is defined by
\[ a(\eta) \star a(\mu) = a(\eta\mu), \]
where $ a : N \rightarrow X $ is the bijection defined by $ a(\eta) = \eta^{-1}[\bar{e}] $. The corresponding bijection $ a' : N' \rightarrow X $ is given by
\begin{eqnarray*}
a'(\varphi_{X} \eta \varphi_{X}^{-1}) & = & \varphi_{X} \eta^{-1} \varphi_{X}^{-1} [\bar{e}]  \\
& = & \varphi_{X} \eta^{-1} [\bar{e}] \\
& = & \varphi_{X} a(\eta),
\end{eqnarray*}
and the binary operation $ \star' $ on $ X $ is given by
\[ a'(\varphi_{X} \eta \varphi_{X}^{-1}) \star' a'(\varphi_{X} \mu \varphi_{X}^{-1}) = a'(\varphi_{X} \eta \mu \varphi_{X}^{-1}). \]
Rewriting in terms of $ a $ we find
\begin{eqnarray*}
\varphi_{X} a(\eta) \star' \varphi_{X} a(\mu) & = & \varphi_{X} a(\eta\mu) \\
& = & \varphi_{X}( a(\eta) \star a(\mu) ),
\end{eqnarray*}
and so $ \varphi_{X} : (X, \star) \rightarrow (X,\star') $ is a homomorphism. Since $ \varphi $ is an automorphism of $ G $ that stabilizes $ G' $ we see that $ \varphi_{X} $ is a bijection, hence an isomorphism. Therefore $ \boldvarphi : (G,\cdot,X,\star,\odot) \rightarrow (G,\cdot,X,\star',\odot) $ is an isomorphism.
\end{proof}

\begin{corollary}
The number of $ G $-stable regular subgroups of $ \Perm(X) $ that arising from an isomorphism class of skew bracoids is
\[ \frac{ | \Aut_{G'}(G) | }{ | \Aut_{G',\star}(G) | }, \]
where $ \Aut_{G'}(G)  = \{ \varphi \in \Aut(G) \mid \varphi(G')=G' \} $ and $ \Aut_{G',\star}(G) $ is the subgroup of $  \Aut_{G'}(G) $ consisting of those automorphisms $ \varphi $ for which $ \varphi_{X} \in \Aut(X,\star) $. 
\end{corollary}

Next we turn to the properties of the Hopf-Galois structure corresponding to a $ G $-stable regular subgroup $ N $ of $ \Perm(X) $. 
Greither and Pareigis show that the Hopf algebra giving the corresponding Hopf-Galois structure is $ \widetilde{L}[N]^{J} $, where $ J $ acts on $ \widetilde{L} $ via the Galois action and on $ N $ via $ \,^{g}\eta = \lambda_{\odot}(g) \eta \lambda_{\odot}(g)^{-1} $, and that the action of this Hopf algebra on $ L $ is given by 
\begin{equation}
\left( \sum_{\eta \in N} c_{\eta} \eta \right)[t] = \sum_{\eta \in N} c_{\eta} \eta^{-1}(eJ')[t]. 
\end{equation}

In the case that $ L/K $ is Galois with Galois group $ (J,\cdot) $, Stefanello and Trappeniers reinterpret this by showing that the Hopf algebra giving Hopf-Galois structure corresponding to a binary operation $ \star $ on $ J $ can be written as $ L[J,\star]^{J} $, with $ J $ acting on $ L $ via the Galois action an on $ (J,\star) $ via the $ \gamma $-function of the skew brace, and that the action of this Hopf algebra on $ L $ is given by 
\begin{equation}
\left( \sum_{j \in J} c_{j} j \right)[t] = \sum_{j \in J} c_{j} j[t]. 
\end{equation}

We generalize this formulation to skew bracoids. 

\begin{proposition} \label{prop_Hopf_algebra_from_SBO}
Let $ \star $ be a binary operation on $ X $ such that $ (G,\cdot,X,\star,\odot) $ is a skew bracoid. The Hopf algebra giving the corresponding Hopf-Galois structure on $ L/K $ is $ E[X,\star]^{G} $, where $ G $ acts on $ E $ as Galois automorphisms and on $ (X,\star) $ via the $ \gamma $-function of the skew bracoid. The action of this $ K $-Hopf algebra on $ L $ is given by
\begin{equation} \label{eqn_Hopf_algebra_action_skew_bracoid}
\left( \sum_{ \bar{x} \in X } c_{\bar{x}} \bar{x} \right)[t] = \sum_{ \bar{x} \in X } c_{\bar{x}} \bar{x}(t) \mbox{ for all } t \in L.
\end{equation}
\end{proposition}
\begin{proof}
By Theorem \ref{thm_Hopf_Galois_structures} that $ \rho_{\star}(X) $ is $ G $-stable regular subgroup of $ \Perm(X) $;  the Hopf algebra giving the corresponding Hopf-Galois structure is $ E[\rho_{\star}(X)]^{G} $. The isomorphism of groups $ \rho_{\star} : (X,\star) \rightarrow \rho_{\star}(X) $ yields an isomorphism of $ E $-Hopf algebras $ E[X,\star] \cong E[\rho_{\star}(X)] $. By Equation \eqref{eqn_GP_action_on_rho_X} we have
\[ \rho_{\star}( \,^{\gamma(g)} \bar{x} ) = \lambda_{\odot}(g) \rho_{\star}(\bar{x}) \lambda_{\odot}(g^{-1}) \]
for all $ g \in G $ and $ \bar{x} \in X $, so this isomorphism is $ G $-equivariant. Therefore by Galois descent we obtain $ E[X,\star]^{G} \cong E[\rho_{\star}(X)]^{G} $ as $ K $-Hopf algebras. 

By the theorem of Greither and Pareigis the the action of the Hopf algebra $ E[\rho_{\star}(X)]^{G} $ on $ L $ is given by
\begin{eqnarray*}
\left( \sum_{ \bar{x} \in X } c_{\bar{x}} \rho_{\star}(\bar{x}) \right)[t] & = & \sum_{ \bar{x} \in X } c_{\bar{x}} \rho_{\star}(\bar{x})^{-1}(\bar{e})[t] \\
& = & \sum_{ \bar{x} \in X } c_{\bar{x}} \bar{x}[t] \mbox{ for all } t \in L. 
\end{eqnarray*}
(Note that the expression $ \bar{x}[t] $ is well defined because $ \bar{x} = xG' $ and $ t \in L = E^{G'} $.) Transporting this to an action of $ E[X,\star]^{G} $ on $ L $ via the inverse of the isomorphism $ E[X,\star]^{G} \cong E[\rho_{\star}(X)]^{G} $, we find
\[ \left( \sum_{ \bar{x} \in X } c_{\bar{x}} \bar{x} \right)[t] = \sum_{ \bar{x} \in X } c_{\bar{x}} \bar{x}(t) \mbox{ for all } t \in L, \]
as claimed.
\end{proof}

The connection between skew braces and Hopf-Galois structures on Galois extensions has been fruitfully applied to questions concerning the \textit{Hopf-Galois correspondence}. If a Hopf algebra $ H $ gives a Hopf-Galois structure on an extension of fields $ L/K $ then each Hopf subalgebra $ H' $ of $ H $ has a corresponding ``fixed field" 
\begin{equation} \label{eqn_fixed_field}
L^{H'} = \{ x \in L \mid h(x) = \varepsilon(h)x \mbox{ for all } h \in H' \}, 
\end{equation}
where $ \varepsilon : H \rightarrow K $ is the counit map of $ H $ (see \cite{CS69} or \cite[Chapter 7]{HAGMT}). The resulting correspondence between Hopf subalgebras of $ H $ and intermediate fields of the extension $ L/K $ is inclusion reversing and injective, but not surjective in general. We say that an intermediate field having the form $ L^{H'} $ for some Hopf subalgebra $ H' $ of $ H $ is \textit{realizable with respect to $ H $}. 

In analogy with the usual Galois correspondence, we always obtain a natural Hopf-Galois structure on the extension $ L/L^{H'} $, given by the $ L^{H'} $-Hopf algebra $ L^{H'} \otimes_{K} H' $. The existence of a quotient Hopf algebra, and a natural quotient Hopf-Galois structure on $ L^{H'}/K $, depend upon whether $ H' $ is a \textit{normal} Hopf subalgebra of $ H $. 

Turning to the Hopf algebras arising in Greither-Pareigis theory, it is well known that the Hopf subalgebras of a group algebra $ E[N] $ are precisely the sets $ E[P] $ with $ P $ a subgroup of $ N $, and that $ E[P] $ is a normal Hopf subalgebra of $ E[N] $ if and only if $ P $ is a normal subgroup of $ N $ (in this case the quotient Hopf algebra identifies naturally with $ E[N/P] $). If $ G $ acts on $ N $ by automorphisms and on $ E $ via the Galois action then by the theory of Galois descent the Hopf subalgebras of $ E[N]^{G} $ are precisely the sets $ E[P]^{G} $ with $ P $ a subgroup of $ N $ that is stable under the action of $ G $, with normality of this Hopf subalgebra again equivalent to normality of $ P $ in $ N $ (in this case we obtain a natural action of $ G $ on $ N/P $, and the quotient Hopf algebra identifies naturally with $ E[N/P]^{G} $). See \cite{ST22} for a recent exposition of these ideas.

In the case that $ L/K $ is Galois with Galois group $ (J,\cdot) $, Childs \cite{Ch18} (using an earlier version of the connection between Hopf-Galois structures and skew braces) characterizes the Hopf subalgebras of a Hopf algebra of the form $ L[N]^{J} $ (and hence the intermediate fields realizable with respect to the corresponding Hopf-Galois structure) in terms of certain substructures of the corresponding skew brace. In the Stefanello-Trappeniers formulation, the Hopf subalgebras of a Hopf algebra of the form $ L[J,\star]^{J} $ correspond to left ideals of $ (J,\star,\cdot) $, with the normal Hopf subalgebras corresponding to ideals. Moreover, they show that if $ I $ is a left ideal of $ (J,\star,\cdot) $ then the fixed field of the Hopf subalgebra $ L[I,\star]^{J} $ coincides with the fixed field $ L^{I} $ obtained by viewing $ I $ as a subgroup of the Galois group. 

We generalize this formulation to skew bracoids. First we note that the left ideals of a skew bracoid of the form $ (G,X) $ have a particular form:

\begin{proposition} \label{prop_ideals_of_G_X}
Let $ Y $ be a left ideal of the skew bracoid $ (G,X) $. Then $ Y = G_{Y}/G' $. 
\end{proposition}
\begin{proof}
Recall from Proposition \ref{prop_left_ideal_sub_SBO} that 
\begin{eqnarray*}
G_{Y} & = & \{ g \in G \mid g \odot \bar{y} \in Y \mbox{ for all } \bar{y} \in Y \} \\
& = & \{ g \in G \mid g \odot \bar{e} \in Y \},
\end{eqnarray*}
and that $ (G_{Y},Y) $ is a subskew bracoid of $ (G,X) $. In this case the action $ \odot $ of $ G $ on $ X $ is simply left translation of cosets, so $ g \odot \bar{e} = \bar{g} $ for each $ g \in G $, and so we have 
\[ Y = G_{Y} \odot \bar{e} = \{ \bar{g} \mid g \in G_{Y} \} = G_{Y} / G'. \]
\end{proof}


\begin{theorem} \label{thm_HG_correspondence}
Let $ (G,\cdot,X,\star,\odot) $ be a skew bracoid and let $ H = E[X,\star]^{G} $ give the corresponding Hopf-Galois structure on $ L/K $.
\begin{enumerate}
\item There is a bijection between intermediate fields $ F $ of $ L/K $ that are realizable with respect to $ H $  and left ideals $ Y $ of $ (G,X) $.
\item Writing $ L^{Y} $ for the intermediate field corresponding to a left ideal $ Y $, we have $ L^{Y} = E^{G_{Y}} $. 
\item The Hopf-Galois structure given by $ H $ on $ L/K $ yields a quotient Hopf-Galois structure on $ L^{Y} / K $ if and only if $ Y $ is an ideal of $ (G,X) $. 
\item The extension $ L^{Y} / K $ is a Galois extension if and only if $ Y $ is an enhanced left ideal of $ (G,X) $. 
\end{enumerate} 
\end{theorem}
\begin{proof}
\begin{enumerate}
\item By the discussion above there is a bijection between intermediate fields $ F $ of $ L/K $ in the image of the Hopf-Galois correspondence with respect to $ H $ and Hopf subalgebras of $ H $. By Proposition \ref{prop_Hopf_algebra_from_SBO} we have $ H = E[X,\star]^{G} $ with $ G $ acting on $ E $ as Galois automorphisms and on $ X $ via $ \gamma $. By the discussion above the Hopf subalgebras of $ E[X,\star]^{G} $ are the subgroups of $ (X,\star) $ that are stable under $ \gamma(G) $, which are precisely the left ideals of $ (G,X) $ (see Definition \ref{defn_left_ideal_and_ideal}).
\item If $ Y $ is a left ideal of $ (G,X) $ then the corresponding Hopf subalgebra of $ E[X,\star]^{G} $ is $ E[Y,\star]^{G} $, and the corresponding fixed field is $ L^{Y} = L^{E[Y,\star]^{G}} $, defined as in \eqref{eqn_fixed_field}. Recall from Proposition \ref{prop_ideals_of_G_X} that $ Y = G_{Y} / G' $; in particular, we have $ G' \subseteq G_{Y} $, so $ E^{G_{Y}} \subseteq E^{G'} = L $. 

Now using the action of $ E[X,\star]^{G} $ on $ L $ given in Proposition \ref{prop_Hopf_algebra_from_SBO} we see that if $ t \in E^{G_{Y}} $ and $ h = \sum_{\bar{y} \in Y} c_{\bar{y}} \bar{y} \in L^{E[Y,\star]^{G}} $ then 
\[ h(t) = \sum_{\bar{y} \in Y} c_{\bar{y}} \bar{y}(t) = \sum_{\bar{y} \in Y} c_{\bar{y}} t = \varepsilon(h)t, \]
so $ t \in L^{Y} $. Hence $ E^{G_{Y}} \subseteq L^{Y} $. But we have
\[ [L:L^{Y}] = \dim_{K}(E[Y,\star]^{G}) = |Y| = \frac{|G_{Y}|}{|G'|} = \frac{ [E:E^{G_{Y}}] }{ [E:L] } = [L:E^{G_{Y}}]. \]
Therefore $ L^{Y} = E^{G_{Y}} $, as claimed. 
\item As described above, we obtain a quotient Hopf-Galois structure on $ L^{Y} / K $ if and only if $ E[Y,\star]^{G} $ is a normal Hopf subalgebra of $ E[X,\star]^{G} $, which occurs if and only if the left ideal $ Y $ of $ (G,X) $ is also a normal subgroup of $ X $, which is precisely the condition for $ Y $ to be an ideal of $ (G,X) $ (see Definition \ref{defn_left_ideal_and_ideal}). 
\item Since $ L^{Y} = E^{G_{Y}} $, the extension $ L^{Y} / K $ is a Galois extension if and only if the subgroup $ G_{Y} $ attached to the left ideal $ Y $ is a normal subgroup of $ G $, which is precisely the condition for $ Y $ to be an enhanced ideal of $ (G,X) $ (see Definition \ref{defn_enhanced_left_ideal}). 
\end{enumerate}
\end{proof}

Finally, we consider the Hopf algebras giving Hopf-Galois structures on the various subextensions. 

In \cite{ST22} it is shown that if $ (J,\star,\cdot) $ is a skew brace with corresponding Hopf-Galois structure $ L[J,\star]^{J} $ on a Galois extension $ L/K $ with Galois group $ (J,\cdot) $ then, given a left ideal $ I $ of $ (J,\star,\cdot) $, the skew brace $ (I,\star,\cdot) $ corresponds to the Hopf-Galois structure given by $ L^{I} \otimes_{K} L[I,\star]^{J} $ on $ L/L^{I} $. Moreover, if $ I $ is an ideal then the skew brace $ (J/I, \star, \cdot) $ corresponds to the Hopf-Galois structure given by $ L[J/I,\star]^{J} $ on $ L^{I}/K $. It is also observed that the Hopf algebra $ L[J/I,\star]^{J} $ gives a Hopf-Galois structure on $ L^{I}/K $ in the case that $ I $ is a strong left ideal, but not an ideal, of $ (J,\star,\cdot) $. In this case $ L^{I}/K $ in a non-normal extension, so there is no corresponding skew brace. 

From the details of the proof of Theorem \ref{thm_HG_correspondence} we see that if $ (G,\cdot,X,\star,\odot) $ is a skew bracoid with corresponding Hopf-Galois structure $ E[X,\star]^{G} $ on a separable extension $ L/K $ and $ Y $ is a left ideal of $ (G,X) $ then the skew bracoid $ (G_{Y},Y) $ corresponds to the Hopf-Galois structure given by $ L^{Y} \otimes_{K} E[Y,\star]^{G} $ on $ L/L^{Y} $. Moreover, if $ Y $ is an ideal then the skew bracoid $ (G, X/Y) $ corresponds to the Hopf-Galois structure given by $ E[X/Y,\star]^{G} $ on $ L^{Y}/K $. Note that this description is valid regardless of whether $ L^{Y}/K $ is a Galois extension or not. In particular, returning to the Galois case for a moment, the Hopf-Galois structure given by $ L[J/I,\star]^{J} $ on the non-normal extension $ L^{I}/K $ in the case that $ I $ is a strong left ideal, but not an ideal, of $ (J,\star,\cdot) $ corresponds under our theory with the skew bracoid $ (J,\cdot,J/I,\star,\odot) $ (a quotient of a skew brace as described in Proposition \ref{prop_skew_bracoid_quotient_skew_brace}. 

In fact, we have: 

\begin{proposition}
Let $ (G,\cdot,X,\star,\odot) $ be a skew bracoid and let $ H = E[X,\star]^{G} $ give the corresponding Hopf-Galois structure on $ L/K $. The following are equivalent:
\begin{enumerate}[ref=\roman*]
\item the Hopf-Galois structure given by $ H $ arises as the quotient of a Hopf-Galois structure on the Galois extension $ E/K $ by some normal Hopf subalgebra; \label{enum_prop_quotient_galois_HGS_1}
\item $ (G,X) $ arises as the quotient of a skew brace by a strong left ideal. \label{enum_prop_quotient_galois_HGS_2}
\end{enumerate}
\end{proposition}
\begin{proof}
First suppose that \eqref{enum_prop_quotient_galois_HGS_1} holds. We may express the Hopf-Galois structure on $ E/K $ as $ E[G,\star]^{G} $ for some binary operation $ \star $ on $ G $ such that $ (G,\star,\cdot) $ is a skew brace, and the normal Hopf subalgebra as $ E[G',\star]^{G} $ with $ (G',\star,\cdot) $ a strong left ideal of $ (G,\star,\cdot) $ Note that the underlying set of this strong left ideal is necessarily equal to $ G' $, since the corresponding fixed field is assumed to be $ L $. The resulting quotient Hopf-Galois structure on $ L/K $ (which coincides with $ H $ by hypothesis) is given by $ E[G/G',\star]^{G} $, where $ (G/G',\star) $ is a quotient group of $ (G,\star) $. By Theorem \ref{thm_Hopf_Galois_structures} this Hopf-Galois structure corresponds to the skew bracoid $ (G,\cdot,X,\star,\odot) $; hence this skew bracoid arises as the quotient of a skew brace by a strong left ideal, and so \eqref{enum_prop_quotient_galois_HGS_2} holds.
Conversely, suppose that \eqref{enum_prop_quotient_galois_HGS_2} holds. Then there exists some binary operation $ \star $ on $ G $ such that $ (G,\star,\cdot) $ is a skew brace, and $ G' $ is a strong left ideal, so that $ (G,X) $ is obtained via the process described in Proposition \ref{prop_skew_bracoid_quotient_skew_brace}. Note that the underlying sets of the skew brace and strong left ideal are necessarily equal to $ G $ and $ G' $, since the sets appearing in the skew bracoid are $ G $ and $ X=G/G' $. The skew brace corresponds to a Hopf-Galois structure on $ E/K $, given by $ E[G,\star]^{G} $, and the strong left ideal corresponds to a normal Hopf subalgebra $ E[G',\star]^{G} $. The resulting quotient Hopf-Galois structure is given by $ E[G/G',\star]^{G} = E[X,\star]^{G} $, which coincides with the Hopf-Galois structure that corresponds to the skew bracoid $ (G,\cdot,X,\star,\odot) $; hence   
\eqref{enum_prop_quotient_galois_HGS_1} holds.
\end{proof}

\bibliographystyle{plain} 
\bibliography{Omahabib}

\end{document}